\documentclass[12 pt]{amsart}
\usepackage{amsmath,times,epsfig,amssymb,amsbsy,amscd,amsfonts,amstext,color,bm}
\usepackage[arrow,matrix]{xy}

\usepackage{graphicx}

\theoremstyle{plain}
\newtheorem{theorem}{Theorem}[section]
\newtheorem{corollary}[theorem]{Corollary}
\newtheorem{lemma}[theorem]{Lemma}

\newtheorem{proposition}[theorem]{Proposition}

\makeatletter
    
    \@addtoreset{equation}{section}
  \makeatother

\theoremstyle{definition}
\newtheorem{definition}[theorem]{Definition}
\newtheorem{example}[theorem]{Example}

\theoremstyle{remark}
\newtheorem{remark}[theorem]{Remark}

\newcommand{\A}{\mathcal{A}}
\newcommand{\scB}{\mathcal{B}}
\newcommand{\C}{\mathbb{C}}
\newcommand{\F}{\mathbb{F}}

\newcommand{\Q}{\mathbb{Q}}
\newcommand{\R}{\mathbb{R}}
\newcommand{\scS}{\mathcal{S}}
\newcommand{\scT}{\mathcal{T}}

\newcommand{\Z}{\mathbb{Z}}

\newcommand{\M}{\mathcal{M}}

\newcommand{\Ker}{\operatorname{Ker}}
\newcommand{\lcm}{\operatorname{lcm}}
\newcommand{\Hom}{\operatorname{Hom}}
\newcommand{\quasi}{\operatorname{quasi}}

\newcommand{\semi}{\operatorname{semi}}
\renewcommand{\top}{\operatorname{top}}

\newcommand{\arith}{\operatorname{arith}}

\newcommand{\free}{\operatorname{free}}
\newcommand{\tor}{\operatorname{tor}}
\newcommand{\Tor}{\operatorname{Tor}}
\newcommand{\rank}{\operatorname{rank}}
\newcommand{\torus}{\operatorname{torus}}
\newcommand{\mer}{\operatorname{merid}}

\begin{document}

\title[$G$-Tutte]{$G$-Tutte polynomials and abelian Lie group arrangements}

\date{\today}

\begin{abstract}
We introduce and study the notion of the $G$-Tutte polynomial for 
a list $\mathcal{A}$ of elements in a finitely generated abelian group $\Gamma$ and 
an abelian group $G$, which is defined 
by counting 
the number of homomorphisms from associated finite abelian groups 
to $G$. 

The $G$-Tutte polynomial is a common generalization of the (arithmetic) 
Tutte polynomial for realizable (arithmetic) matroids, 
the characteristic quasi-polynomial for integral arrangements, 
Br\"and\'en-Moci's arithmetic version of the partition function 
of an abelian group-valued Potts model, and 
the modified Tutte-Krushkal-Renhardy polynomial for a finite CW-complex. 

As in the classical case, $G$-Tutte polynomials carry 
topological and enumerative information (e.g., the Euler characteristic, 
point counting and the Poincar\'e polynomial) 
of abelian Lie group arrangements. 

We also discuss differences between the arithmetic Tutte and the 
$G$-Tutte polynomials related to the axioms for arithmetic matroids and the 
(non-)positivity of coefficients. 
\end{abstract}

\author{Ye Liu}
\address{Ye Liu, Department of Mathematics, Hokkaido University, Kita 10, Nishi 8, Kita-Ku, Sapporo 060-0810, Japan.}
\email{liu@math.sci.hokudai.ac.jp}
\author{Tan Nhat Tran}
\address{Tan Nhat Tran, Department of Mathematics, Hokkaido University, Kita 10, Nishi 8, Kita-Ku, Sapporo 060-0810, Japan.}
\email{trannhattan@math.sci.hokudai.ac.jp}
\author{Masahiko Yoshinaga}
\address{Masahiko Yoshinaga, Department of Mathematics, Faculty of Science, Hokkaido University, Kita 10, Nishi 8, Kita-Ku, Sapporo 060-0810, Japan.}
\email{yoshinaga@math.sci.hokudai.ac.jp}


\keywords{Tutte polynomial, characteristic quasi-polynomial, Poincare polynomial, arithmetic matroids}

\date{\today}
\maketitle

\tableofcontents

\section{Introduction}
\label{sec:intro}

The Tutte polynomial is one of the most important invariants 
of a graph. The significance of the Tutte polynomial 
is that it has several important specializations, including 
chromatic polynomials, partition functions of Potts models (\cite{sokal}), 
and Jones polynomials for alternating links (\cite{thi}). 
Another noteworthy aspect of the Tutte polynomial 
is that it depends only on the (graphical) matroid structure, 
and thus one can define the Tutte polynomial for a matroid. 
Matroids and (specializations of) Tutte polynomials play a role 
in several diverse areas of mathematics (\cite{oxley, welsh}). 

Matroids and Tutte polynomials 
are particularly important in the study 
of hyperplane arrangements (\cite{ot}), 
because the Tutte polynomial and one of its specializations, 
the characteristic polynomial, 
carry enumerative and topological information about the 
arrangement. For instance, the number of points over a finite field, 
the number of chambers for a real arrangement and the Betti numbers 
for a complex arrangement are all obtained from the characteristic polynomial. 


It is natural to consider arrangements of subsets of other types. 
Recently, arrangements of subtori in a torus, or so-called 
toric arrangements, have received considerable attention 
(\cite{dec-pro}), 
which has origin in the study of the moduli space of curves (\cite{loo-coh}) 
and regular semisimple elements in an algebraic 
group (\cite{leh-tor}). 


The notions of arithmetic Tutte polynomials and arithmetic matroids invented 
by Moci and collaborators 
(\cite{moci-tor, dad-mo, br-mo, fink-moci}) are particularly useful 
for studying toric arrangements. 
As in the case of hyperplane arrangements, arithmetic 
Tutte polynomials carry enumerative and topological information about 
toric arrangements. It is generally difficult to explicitly compute 
the arithmetic Tutte polynomial. 
Arithmetic Tutte polynomials for classical root systems were computed by 
Ardila, Castillo and Henley (\cite{ar-ca-he}). 

Another (quasi-)polynomial invariant for a hyperplane 
arrangement defined over integers, the characteristic quasi-polynomial 
introduced by Kamiya, Takemura and Terao \cite{ktt-cent}, is 
a refinement of the characteristic polynomial of an arrangement. 
The notion of the characteristic quasi-polynomial is closely related to 
Ehrhart theory on counting lattice points, and has increased in combinatorial 
importance recently. 
The characteristic quasi-polynomial for root systems was essentially 
computed by Suter \cite{sut} (see also \cite{ktt-quasi}). 
By comparing the computations of Suter with those 
of Ardila, Castillo and Henley, 
it has been observed that the most degenerate constituent of 
the characteristic quasi-polynomial is a specialization 
of the arithmetic Tutte polynomial. 

The purpose of this paper is to introduce and study a new class of 
polynomial invariant that forms a common generalization of the 
Tutte, arithmetic Tutte and characteristic quasi-polynomials, among others. 

The key observation to unify the above 
``Tutte-like polynomials'' is that they are all defined by means of counting 
homomorphisms between certain abelian groups (this formulation 
appeared in \cite[\S 7]{br-mo}). 
This observation has prompted us to introduce the notion of the 
\emph{$G$-Tutte polynomial} 
$T_{\A}^G(x, y)$ for a list of elements $\A$ in a finitely generated 
abelian group $\Gamma$ and an abelian group $G$ with a certain 
finiteness assumption on the torsion elements (see \S \ref{sec:GTutte} 
for details). We mainly consider abelian Lie groups 
$G$ of the form 
\begin{equation*}
G=F\times(S^1)^p\times \R^q, 
\end{equation*}
where $F$ is a finite abelian group and $p, q\geq 0$. 
Typical examples are $\C\simeq\R^2$ and 
$\C^\times\simeq S^1\times\R$. 

When the group $G$ is $\C$, $\C^\times$, or the finite cyclic group 
$\Z/k\Z$, 
the $G$-Tutte polynomial is precisely the Tutte polynomial, 
the arithmetic Tutte polynomial, or a constituent of the characteristic 
quasi-polynomial, respectively. 
We will see that many known properties (deletion-contraction formula, 
Euler characteristic of the complement, 
point counting, Poincar\'e polynomial, convolution formula) 
for (arithmetic) Tutte polynomials are shared by $G$-Tutte polynomials. 
(See \cite{de-ri} for another attempt to generalize arithmetic 
Tutte polynomials.) 

The organization of this paper is as follows. 

\S \ref{sec:background} gives a summary of background material. 
We recall definitions of the 
Tutte polynomial $T_{\A}(x, y)$, arithmetic Tutte polynomial 
$T_{\A}^{\arith}(x, y)$ and the characteristic quasi-polynomial 
$\chi_{\A}^{\quasi}(q)$ for a given list of elements $\A$ in 
$\Gamma=\Z^\ell$. 

As pointed out by D'Adderio-Moci \cite{dad-mo}, it is more convenient to 
consider a list $\A$ in a finitely generated abelian group $\Gamma$. 
Following their ideas, in \S \ref{sec:arr} we define 
arrangements $\A(G)$ of subgroups in $\Hom(\Gamma, G)$ 
and its complements $\M(\A; \Gamma, G)$ for arbitrary 
abelian group $G$. 
We see that the set-theoretic deletion-contraction formula holds. 

In \S \ref{sec:GTutte}, the $G$-Tutte polynomial $T_{\A}^G(x, y)$ 
is defined using the number of homomorphisms 
of certain finite abelian groups to $G$ (the $G$-multiplicities). 
We also define the multivariate version $Z_{\A}^G(q, \bm{v})$ and the 
$G$-characteristic polynomial $\chi_{\A}^G(t)$. 
We then show that the $G$-Tutte polynomial satisfies the 
deletion-contraction formula . We also see that the $G$-Tutte polynomial 
has several specializations. 

In \S \ref{sec:eulercount}, we show that the Euler characteristic 
$e(\M(\A; \Gamma, G))$ of the complement can be computed as a special value of the 
$G$-Tutte polynomial (or $G$-characteristic polynomial) when $G$ is 
an abelian Lie group with finitely many components. 
As a special case, when $G$ is finite, we obtain a formula that counts 
the cardinality $\#\M(\A; \Gamma, G)$. The equality between 
the arithmetic characteristic polynomial and the most degenerate 
constituent of the characteristic quasi-polynomial is also proved. 

In \S \ref{sec:example} we compute the Poincar\'e polynomial for 
toric arrangements associated with root systems (considering positive roots 
to be a list in the root lattice). 
Applying recent results on characteristic 
quasi-polynomials, we show that the Poincar\'e polynomial satisfies a certain 
self-duality when the root system differs from $E_7, E_8$. 
We also recover Moci's results on Euler characteristics 
\cite[Corollary 7.3]{moci-tor}. 

In \S \ref{sec:poinc}, we prove a formula that expresses 
the Poincar\'e polynomial 
of $\M(\A; \Gamma, G)$ in terms of $G$-characteristic polynomials under the assumption 
that $G$ is a non-compact abelian Lie group with finitely many connected components. 
This formula covers several classical results, including 
hyperplane arrangements (Orlik-Solomon \cite{os} and 
Zaslavsky \cite{zas-face}), 
certain subspace arrangements 
(Goresky-MacPherson \cite{gor-mac}, Bj\"orner \cite{bjo}) and toric 
arrangements 
(De Concini-Procesi \cite{dec-pro}, Moci \cite{moci-tor}). 

If $G=S^1$ or $\C^\times$, then the $G$-multiplicities satisfy the 
five axioms of arithmetic matroids given in \cite{dad-mo}. 
A natural question to ask is whether the 
$G$-multiplicities satisfy these axioms for general groups $G$. 
In \S \ref{sec:matroid}, we show that four of the five axioms 
are satisfied by the $G$-multiplicities. 
However, one of the 
axioms is not necessarily satisfied. We also present a counter-example that 
does not satisfy this axiom. 
However, 
we prove that the $G$-multiplicity function 
satisfies another important formula, 
the so-called convolution formula, which has been a formula 
of interest recently \cite{ba-le, de-mo}. 
Unlike the cases of arithmetic Tutte polynomials, the coefficients 
of $G$-Tutte polynomials are not necessarily positive. We show this 
with an example. 

For the purpose of giving a combinatorial framework that 
describes intersection patterns of arrangements 
over an abelian Lie group $G$, 
an interesting problem would be to axiomatize $G$-multiplicities, 
which is left for future research. 

\medskip

\noindent
\emph{Conventions}: 
In this paper, the term {\bf list} is synonymous with multiset. 
We follow the convention in \cite[\S 2.1]{dad-mo}. For example, 
the list $\A=\{\alpha, \alpha\}$ has $4$ distinct sublists: 
$\scS_1=\emptyset, 
\scS_2=\{\alpha\},  
\scS_3=\{\alpha\},  
\scS_4=\{\alpha, \alpha\}=\A$. We distinguish $\scS_2$ and $\scS_3$, 
and hence $\A\smallsetminus\scS_2=\scS_3$. 
If $\A$ is a list, then $\scS\subset\A$ indicates 
that $\scS$ is a sublist of $\A$. 

(In \S \ref{subsec:finiteGamma} and \S \ref{sec:matroid}) 
A dot under a letter 
indicates the parameter in the summation. 
For instance, $\sum_{\scS\subset\underset{\bullet}{\scB}\subset\scT}$ 
indicates that 
$\scS$ and $\scT$ are fixed, and $\scB$ is running between them.

\section{Background}

\label{sec:background}

In this section, we recall the definitions of Tutte polynomials, 
arithmetic Tutte polynomials, 
characteristic quasi-polynomials and related results.

\subsection{(Arithmetic) Tutte polynomials}

\label{subsec:arithmtutte}

Let $\A=\{\alpha_1, \dots, \alpha_n\}\subset \Z^\ell$ be a list of 
integer vectors, let $\alpha_i=(a_{i1}, \dots, a_{i\ell})$. 
We may consider $\alpha_i$ to be a linear form defined by 
\begin{equation*}
\alpha_i(x_1, \dots, x_\ell)=a_{i1}x_1+\cdots +a_{i\ell}x_\ell. 
\end{equation*}
A sublist $\scS\subset\A$ determines a homomorphism 
$\alpha_\scS:\Z^\ell\longrightarrow\Z^{\#\scS}$. 

Let $G$ be an abelian group. 
Define the subgroup $H_{\alpha_i, G}$ of $G^\ell$ by 
\begin{equation*}
H_{\alpha_i, G}=\Ker(\alpha_i\otimes G: G^\ell\longrightarrow G). 
\end{equation*}
The list $\A$ determines an arrangement 
$\A(G)=\{H_{\alpha, G}\mid\alpha\in\A\}$ 
of subgroups in $G^\ell$. 
Denote their complement by 
\begin{equation*}
\M(\A; \Z^\ell, G):=G^\ell\smallsetminus
\bigcup_{\alpha_i\in\A}H_{\alpha_i, G}. 
\end{equation*}
The arrangement $\A(G)$ of subgroups and its complement 
$\M(\A; \Z^\ell, G)$ are important objects of study in many contexts. 
We list some of these below. 
\begin{itemize}
\item[(i)] 
When $G$ is the additive group of a field (e.g., $G=\C, \R, \F_q$), 
$\A(G)$ is the associated hyperplane arrangement (\cite{ot}). 
\item[(ii)] 
When $G=\R^c$ with $c>0$, $\A(G)$ is called the $c$-plexification of 
$\A$ (see \cite[\S 5.2]{bjo}). 
\item[(iii)] 
When $G$ is $\C^\times$ or $S^1$, $\A(G)$ is called a toric arrangement. 
\item[(iv)] 
When $G=S^1\times S^1$ (viewed as an elliptic curve), $\A(G)$ is called an 
elliptic (or abelian) arrangement. (\cite{bibby}). 
\item[(v)] 
When $G$ is a finite cyclic group $\Z/q\Z$, $\A(G)$ is related to the 
characteristic quasi-polynomial studied in 
\cite{ktt-cent, ktt-quasi} (see \ref{subsec:cqp}). 
There is also an important connection with Ehrhart theory and enumerative 
problems (\cite{bl-sa, yos-wor, yos-linial}). 
\end{itemize}
To define the arithmetic Tutte polynomial, we need further notation. 
The linear map $\alpha_\scS$ is expressed by 
the matrix $M_{\scS}=(a_{ij})_{i\in\scS, 1\leq j\leq\ell}$ of size 
$\#\scS\times \ell$. 
Denote by $r_\scS$ the rank of $M_{\scS}$. 
Suppose that the Smith normal form of $M_{\scS}$ is 
\begin{equation*}
\begin{pmatrix}
d_{\scS, 1}&0&\cdots&0&\cdots&\cdots&0\\
0&d_{\scS, 2}& &\vdots & & &\vdots\\
\vdots& &\ddots&0& & & \\
0&\cdots&0& d_{\scS, r_\scS}&&&\\
\vdots& & & &0&& \\
\vdots& & & & &\ddots&\vdots\\
0&\cdots& & & &\cdots&0
\end{pmatrix}, 
\end{equation*}
where $1\leq d_{\scS, i}$ is a positive integer 
and $d_{\scS, i}$ divides $d_{\scS, i+1}$. 

The Tutte polynomial $T_{\A}(x, y)$ and the 
arithmetic Tutte polynomial $T_{\A}^{\arith}(x, y)$ 
of $\A$ are defined as follows (\cite{moci-tor, br-mo}). 
\begin{equation*}
\begin{split}
T_{\A}(x, y)&=\sum_{\scS\subset\A}(x-1)^{r_\A-r_\scS}(y-1)^{\#\scS-r_\scS}, \\
T_{\A}^{\arith}(x, y)&=\sum_{\scS\subset\A}m(\scS)(x-1)^{r_\A-r_\scS}
(y-1)^{\#\scS-r_\scS}, 
\end{split}
\end{equation*}
where $m(\scS)=\prod_{i=1}^{r_\scS}d_{\scS, i}$. 
($m(\scS)$ can also be defined as the cardinality of 
torsion subgroup of $\Z^\ell$ quotient by the subgroup generated by 
the row vectors of $M_{\scS}$. 
See also \S \ref{subsec:gtutte}). It should be noted that 
the arithmetic Tutte polynomial $T_{\A}^{\arith}(x, y)$ is defined for 
more general objects, arithmetic matroids \cite{dad-mo} and 
matroids over $\Z$ \cite{fink-moci}.

These polynomials encode combinatorial and topological information about 
the arrangements. For instance, the characteristic polynomial 
of the ranked poset of flats of the hyperplane arrangement is 
$\chi_\A(t)=(-1)^{r_\A}t^{\ell-r_\A}T_{\A}(1-t,0)$, and 
the Poincar\'e polynomial 
of $\M(\A; \Z^\ell, \R^c)$ is (\cite{gor-mac, bjo}) 
\begin{equation}
\label{eq:bjorner}
P_{\M(\A; \Z^\ell, \R^c)}(t)=t^{r_\A\cdot (c-1)}\cdot T_\A\left(\frac{1+t}{t^{c-1}}, 
0\right). 
\end{equation}
Note that the special cases $c=1$ and $c=2$ reduce to the famous formulas 
given by 
Zaslavsky \cite{zas-face} and Orlik-Solomon \cite{os}, respectively. 
Similarly, as proved by De Concini-Procesi \cite{dec-pro} and Moci 
\cite{moci-tor}, 
the characteristic polynomial of the layers (connected components of 
intersections) of the corresponding toric arrangement is 
$\chi_\A^{\arith}(t)=(-1)^{r_\A}t^{\ell-r_\A}T_{\A}^{\arith}(1-t,0)$, and 
the Poincar\'e polynomial of $\M(\A; \Z^\ell, \C^\times)$ is 
\begin{equation}
\label{eq:torpoin}
P_{\M(\A; \Z^\ell, \C^\times)}(t)=
(1+t)^{\ell-r_\A}\cdot 
t^{r_\A}\cdot T_\A^{\arith}\left(\frac{1+2t}{t}, 0\right). 
\end{equation}
The cohomology ring structure of $\M(\A; \Z^\ell, \C^\times)$ was recently described in \cite{cddmp}.

Contrary to the above cases, 
Bibby \cite[Remark 4.4]{bibby} pointed out that 
when $G=S^1\times S^1$, 
a similar formula for the Poincar\'e polynomial 
does not hold. 
We will see that the non-compactness of $G$ plays an important role in 
Poincar\'e polynomial formulas (Theorem \ref{thm:poin} and 
Remark \ref{rem:noncpt}). 

\subsection{Characteristic quasi-polynomials}
\label{subsec:cqp}

In \cite{ktt-cent}, 
Kamiya, Takemura and Terao proved that 
$\#\M(\A; \Z^\ell, \Z/q\Z)$ is a quasi-polynomial in $q$ ($q\in\Z_{>0}$), 
denoted by $\chi^{\quasi}_{\A}(q)$, with period 
\begin{equation*}
\rho_\A:=\lcm (d_{\scS, r_{\scS}}\mid \scS\subset\A). 
\end{equation*}
More precisely, there exist polynomials 
$f_1(t), f_2(t), \cdots, f_{\rho_\A}(t)\in\Z[t]$ such that 
for any positive integer $q$, 
\begin{equation*}
\chi^{\quasi}_{\A}(q):=\#\M(\A; \Z^\ell, \Z/q\Z) =f_k(q), 
\end{equation*}
where $k\equiv q\mod \rho_\A$. 
The polynomial $f_k(t)$ is called the $k$-constituent. 
They also proved that $f_k(t)=f_m(t)$ if $\gcd(k, \rho_\A)=
\gcd(m, \rho_\A)$. 
Furthermore, the $1$-constituent $f_1(t)$ (and more generally, 
$f_k(t)$ with $\gcd(k, \rho_\A)=1$) is known to be equal to the 
characteristic polynomial $\chi_\A(t)$ 
(\cite{ath-adv}). 

We will show that the most degenerate constituent $f_{\rho_\A}(t)$ 
is obtained as a specialization of the arithmetic Tutte polynomial, 
and that the other constituents can also be described in terms of 
the $G$-Tutte polynomials introduced later (Theorem \ref{thm:cqp}, 
Corollary \ref{cor:toricchar}). 

\section{Arrangements over abelian groups and deletion-contraction} 

\label{sec:arr}

Let $\Gamma$ be a finitely generated abelian group, 
$\A=\{\alpha_1, \dots, \alpha_n\}\subset\Gamma$ a list 
(multiset) of finitely many elements, 
and $G$ an arbitrary abelian group 
with the unit $e\in G$. 
In this section, 
we define arrangements over $G$, and prove the 
set-theoretic deletion-contraction formula 
by generalizing an idea of D'Adderio-Moci \cite[\S 3.2]{dad-mo}. 

\subsection{$G$-plexification}

\label{subsec:G-plexif}

Let us denote the subgroup of torsion elements of $\Gamma$ by 
$\Gamma_{\tor}\subset\Gamma$, and 
the rank of $\Gamma$ by $r_\Gamma$. 
More generally, for a 
list 
$\scS\subset\Gamma$, denote the rank 

of the subgroup $\langle\scS\rangle\subset\Gamma$ 
generated by $\scS$ by 
\begin{equation*}
r_\scS=\rank \langle\scS\rangle. 
\end{equation*}

We now define the ``arrangement'' associated with a list $\A$ over 
an arbitrary abelian group $G$. 
The total space is the abelian group 
\begin{equation*}
\Hom(\Gamma, G)=\{\varphi :\Gamma\longrightarrow G\mid\mbox{ $\varphi$ is a 
homomorphism}\} 
\end{equation*}
of all homomorphisms from $\Gamma$ to $G$. The element $\alpha\in\Gamma$ 
naturally determines a homomorphism 
\begin{equation*}
\alpha:\Hom(\Gamma, G)\longrightarrow G, \varphi\longmapsto\varphi(\alpha), 
\end{equation*}
with the kernel 
\begin{equation*}
H_{\alpha, G}:=\{\varphi \in\Hom(\Gamma, G)\mid \varphi (\alpha)= e\}. 
\end{equation*}
The collection of subgroups $\A(G)=\{H_{\alpha, G}\mid\alpha\in\A\}$ is 
called the $G$-plexification of $\A$. Denote the complement of $\A(G)$ by 
\begin{equation*}
\M(\A; \Gamma, G):=\Hom(\Gamma, G)\smallsetminus\bigcup_{\alpha\in\A}H_{\alpha, G}. 
\end{equation*}

\begin{example}
(1) 
Suppose that $\Gamma=\Z^\ell$. 
Then $\Hom(\Gamma, G)\simeq G^\ell$. 

(2) 
Suppose that $\Gamma=\Z/d\Z$. 
Then 
\begin{equation*}
\Hom(\Gamma, G)\simeq G[d]:=\{x\in G\mid d\cdot x=e\}
\end{equation*}
is the subgroup of $d$-torsion points. 
\end{example}

\begin{example}
\label{ex:smallex}
Let $\Gamma=\Z\oplus\Z/4\Z$ and $G=\C^\times$. Then 
$\Hom(\Gamma, G)\simeq \C^\times\times \{\pm 1, \pm i\}$, 
which is a (real $2$-dimensional) 
Lie group with $4$ connected components. 
If $\alpha_1=(2, 2)\in\Gamma$, then $H_{\alpha_1, G}=\{(\pm 1, \pm 1), 
(\pm i, \pm i)\}$ 
consists of $8$ points. 
If $\alpha_2=(0, 2)\in\Gamma$, then $H_{\alpha_2, G}=\C^\times\times\{\pm 1\}$ 
is a union of two copies of $\C^\times$. 
\end{example}

\begin{remark}
$G$-plexification can be considered as a generalization of 
complexification, $c$-plexification, toric arrangements and 
$\Z/q\Z$ reduction (see \S \ref{sec:background}). 
\end{remark}

\subsection{Set-theoretic deletion-contraction formula}
\label{subsec:delcont}

D'Adderio-Moci \cite{dad-mo} 
defined two other lists for a fixed $\alpha\in\A$, 
the deletion $\A'$ and the contraction $\A''=\A/\alpha$. 
The deletion is just $\A'=\A\smallsetminus\{\alpha\}$, a 
list of elements in the same group $\Gamma':=\Gamma$. 
To define $\A''$, 
let $\Gamma'':=\Gamma/\langle\alpha\rangle$, and 
$\A'':=\{\overline{\alpha'}\mid\alpha'\in\A'\}\subset
\Gamma''$. By the exact sequence 
\begin{equation*}
0
\longrightarrow\Hom(\Gamma'', G)
\longrightarrow\Hom(\Gamma, G)
\longrightarrow\Hom(\langle\alpha\rangle, G), 
\end{equation*}
the group $\Hom(\Gamma'', G)$ can be 
identified with 
\begin{equation*}
H_{\alpha, G}=
\{\varphi\in\Hom(\Gamma, G)\mid \varphi(\alpha)=e\}. 
\end{equation*}
Therefore, we can consider both 
$\M(\A; \Gamma, G)$ and $\M(\A''; \Gamma'', G)$ as subsets of $\Hom(\Gamma, G)$ 
(actually, the subsets of $\M(\A'; \Gamma', G)$). 
These three sets are related by the following deletion-contraction formula. 
\begin{proposition}
\label{prop:set}
Using the above identification, we have the following decomposition: 
\begin{equation*}
\M(\A'; \Gamma', G)=
\M(\A''; \Gamma'', G)\sqcup\M(\A; \Gamma, G). 
\end{equation*}
\end{proposition}
\begin{proof}
The set $\M(\A'; \Gamma', G)$ can be decomposed as 
$\{\varphi\in\M(\A'; \Gamma', G)\mid \varphi(\alpha)=e\}\sqcup
\{\varphi\in\M(\A'; \Gamma', G)\mid \varphi(\alpha)\neq e\}$. 
The first term on the right-hand side can be 
identified with 
$\{\varphi
\in\M(\A'; \Gamma', G) 
\mid\varphi(\alpha)=e\}\simeq\M(\A''; \Gamma'', G)$, and 
the second term is equal to $\M(\A; \Gamma, G)$. 
\end{proof}

More generally, 
for any sublist $\scS\subset\A$, we can define the contraction 
$\A/\scS$ as the list of cosets $\{\overline{\alpha}\mid 
\alpha\in\A\smallsetminus\scS\}$ in the group $\Gamma/\langle\scS\rangle$. 
Because $\Hom(\Gamma/\langle\scS\rangle, G)$ can be naturally identified with 
the set $\{\varphi\in\Hom(\Gamma, G)\mid \varphi(\alpha)=e, 
\forall \alpha\in\scS\}$, it can be seen as a subset of 
$\Hom(\Gamma, G)$. 
Under this identification, we can describe $\M(\A/\scS; \Gamma/\langle\scS\rangle, G)$ as 
\begin{equation*}
\M(\A/\scS; \Gamma/\langle\scS\rangle, G)=
\left\{
\varphi\in\Hom(\Gamma, G)
\left|
\begin{array}{cc}
\varphi(\alpha)=e, &\mbox{ for }\alpha\in\scS\\
\varphi(\alpha)\neq e, &\mbox{ for }\alpha\in\A\smallsetminus\scS
\end{array}
\right.
\right\}.
\end{equation*}
It is easily seen that for any $\varphi\in\Hom(\Gamma, G)$, 
$\scS=\A_{\varphi}:=\{\alpha\in\A\mid\varphi(\alpha)=e\}$ is the unique 
sublist $\scS\subset\A$ that satisfies $\varphi\in\M(\A/\scS; \Gamma/\langle\scS\rangle, G)$. 
This yields the following. 
\begin{proposition}
\label{prop:disjoint}
Let $\Gamma$ be a finitely generated abelian group, $\A$ be a 
finite 
list of elements in $\Gamma$, and $G$ be an abelian group. 
Then 
\begin{equation*}
\Hom(\Gamma, G)=\bigsqcup_{\scS\subset\A}\M(\A/\scS; \Gamma/\langle\scS\rangle, G). 
\end{equation*}
\end{proposition}

The structure of the 
intersection of the $H_{\alpha, G}$ is described by the next proposition. 

\begin{proposition}
\label{prop:intersect}
Let $\A$ be a finite list of elements in a finitely generated 
abelian group $\Gamma$. Then 
\begin{equation*}
\begin{split}
\bigcap_{\alpha\in\A}H_{\alpha, G}
&\simeq
\Hom(\Gamma/\langle\A\rangle, G)
\\
&\simeq
\Hom((\Gamma/\langle\A\rangle)_{\tor}, G)\times
G^{r_{\Gamma}-r_{\A}}. 
\end{split}
\end{equation*}
\end{proposition}

\begin{proof}
Recall that $H_{\alpha, G}\simeq\Hom(\Gamma/\langle\alpha\rangle, G)$. Hence 
$\bigcap_{\alpha\in\A}H_{\alpha, G}\simeq
\Hom(\Gamma/\langle\A\rangle, G)$. From the structure theorem 
for finitely generated abelian groups, we may assume that 
$\Gamma/\langle\A\rangle\simeq(\Gamma/\langle\A\rangle)_{\tor}\oplus
\Z^{r_{\Gamma}-r_{\A}}$. The result follows from this isomorphism. 
\end{proof}

\section{$G$-Tutte polynomials}

\label{sec:GTutte}

In this section, we define the (multivariate) $G$-Tutte polynomial and 
the $G$-characteristic polynomial for a finite list 
$\A\subset\Gamma$ and an 
abelian group $G$. 
We present the deletion-contraction formulas for these polynomials. 
We also give several specializations of $G$-Tutte polynomials. 

\subsection{Torsion-wise finite abelian groups}
\label{subsec:torsionfin}

\begin{definition}
\label{def:torfin}
An abelian group $G$ is called 
\emph{torsion-wise finite} if the subgroup of $d$-torsion points 
$G[d]$ is finite for all $d>0$. 
\end{definition}

\begin{example}
The following are examples of torsion-wise finite abelian groups. 
\begin{itemize}
\item 
Every torsion-free abelian group (e.g., $\{0\}, \Z, \R, \C$) 
is torsion-wise finite. 
\item 
Every finitely generated abelian group is torsion-wise finite. 
\item 
Every subgroup of 
the multiplicative group $K^\times$ for any field $K$ is 
torsion-wise finite 
(e.g., $(S^1, \times)$ and $(\C^\times, \times)$). 
\end{itemize}
\end{example}

\begin{example}
$(\Z/2\Z)^\infty$ is not a torsion-wise finite group. 
\end{example}

The class of torsion-wise finite groups is closed under 
taking subgroups and finite direct products. 
We mainly study torsion-wise finite groups of the form 
\begin{equation*}
G\simeq F\times (S^1)^p\times \R^q, 
\end{equation*}
where $F$ is a finite abelian group and $p, q\geq 0$. 

\begin{proposition}
\label{prop:finitness}
Let $G$ be a torsion-wise finite abelian group. 
Let $F$ be a finite abelian group. Then 
$\Hom(F, G)$ is finite. 
\end{proposition}
\begin{proof}
By the structure theorem, we may assume that $F\simeq \Z/d_1\Z\times\cdots
\times\Z/d_k\Z$. Then 
\begin{equation*}
\begin{split}
\Hom(F, G)&=\Hom(\Z/d_1\Z\times\cdots
\times\Z/d_k\Z, G)\\
&=
G[d_1]\times\cdots\times G[d_k] 
\end{split}
\end{equation*}
is finite by definition. 
\end{proof}

The next proposition is useful later 
(we omit the proofs). 
\begin{proposition}
\label{prop:finitetf}
(1) $\Hom(\Gamma, G_1\times G_2)\simeq \Hom(\Gamma, G_1)\times 
\Hom(\Gamma, G_2)$. In particular, if $\Hom(\Gamma, G_1\times G_2)$ is finite, 
then $\#\Hom(\Gamma, G_1\times G_2)=\# \Hom(\Gamma, G_1)\times 
\#\Hom(\Gamma, G_2)$. 

(2) Let $d_1, d_2$ be positive integers. Then 
\begin{equation*}
\Hom(\Z/d_1\Z, \Z/d_2\Z)\simeq
\Hom(\Z/d_2\Z, \Z/d_1\Z)\simeq
\Z/\gcd(d_1, d_2)\Z. 
\end{equation*}
In particular, 
$\#\Hom(\Z/d_1\Z, \Z/d_2\Z)=\gcd(d_1, d_2)$. 
\end{proposition}

\subsection{$G$-Tutte polynomials for torsion-wise finite abelian groups}

\label{subsec:gtutte}

\begin{definition}
\label{def:multiplicit}
Let $\Gamma$ be a finitely generated abelian group, 
$\A$ a list of elements in $\Gamma$, and $G$ a torsion-wise finite 
abelian group. 
Define the $G$-multiplicity $m(\A; G)\in\Z$ by 
\begin{equation*}
m(\A; G):=
\#\Hom\left((\Gamma/\langle\A\rangle)_{\tor}, G\right). 
\end{equation*}
(Recall that $(\Gamma/\langle\A\rangle)_{\tor}$ is the torsion part of 
the group $\Gamma/\langle\A\rangle$.) 
\end{definition}
Let us describe $m(\A; G)$ more explicitly. 
Because $\Gamma/\langle\A\rangle$ is a finitely generated abelian group, 
it is isomorphic to a group of the form 
$\Z/d_1\Z\oplus\cdots\oplus\Z/d_r\Z\oplus\Z^{r_\Gamma-r_\A}$. Thus 
$(\Gamma/\langle\A\rangle)_{\tor}\simeq\bigoplus_{i=1}^r\Z/d_i\Z$, and 
\begin{equation*}
\Hom\left((\Gamma/\langle\A\rangle)_{\tor}, G\right)\simeq
\bigoplus_{i=1}^r G[d_i]. 
\end{equation*}
Therefore, $m(\A; G)=\prod_{i=1}^r\#G[d_i]$. 

\begin{remark}
It is also easily seen that 
$\Hom\left((\Gamma/\langle\A\rangle)_{\tor}, G\right)$ is (non-canonically) 
isomorphic to $\Tor_1^{\Z}(\Gamma/\langle\A\rangle, G)$. 
Hence we may define 
$m(\A; G):=\#\Tor_1^{\Z}(\Gamma/\langle\A\rangle, G)$. 
\end{remark}

\begin{definition}
\label{def:main}
Let 
$\A=\{\alpha_1, \dots, \alpha_n\}\subset\Gamma$ be a finite list of elements 
in a finitely generated abelian group $\Gamma$, and 
$G$ a torsion-wise finite group. 
Recall that $r_\Gamma$ and $r_{\scS}$ denote the rank of 
$\Gamma$ and $\langle\scS\rangle$, respectively (\S \ref{subsec:G-plexif}). 
\begin{itemize}
\item[(1)] 
Define the \emph{multivariate $G$-Tutte polynomial} of $\A$ and $G$ by 
\begin{equation*}
Z_{\A}^{G}(q, v_1, \dots, v_n):=
\sum_{\scS\subset\A}m(\scS; G)q^{-r_\scS}\prod_{\alpha_i\in\scS}v_i. 
\end{equation*}
\item[(2)] 
Define the \emph{$G$-Tutte polynomial} of $\A$ and $G$ by 
\begin{equation*}
T_{\A}^{G}(x, y):=
\sum_{\scS\subset\A}m(\scS; G)(x-1)^{r_\A-r_\scS}(y-1)^{\#\scS-r_\scS}. 
\end{equation*}
\item[(3)] 
Define the \emph{$G$-characteristic polynomial} of $\A$ and $G$ by 
\begin{equation*}
\chi_{\A}^G(t):=
\sum_{\scS\subset\A}(-1)^{\#\scS}
m(\scS; G)\cdot t^{r_{\Gamma}-r_\scS}. 
\end{equation*}
\end{itemize}
\end{definition}
These three polynomials are related by the following formulas, as 
in the cases of the Tutte and the arithmetic Tutte polynomials 
(\cite{br-mo, moci-tor, sokal}). 
\begin{equation*}
\begin{split}
T_{\A}^{G}(x, y)
&=
(x-1)^{r_\A}\cdot 
Z_{\A}^{G}((x-1)(y-1), y-1, \dots, y-1), \\
\chi_{\A}^G(t)
&=
(-1)^{r_\A}\cdot t^{r_{\Gamma}-r_\A}\cdot T_{\A}^{G}(1-t, 0). 
\end{split}
\end{equation*}

Recall that $\alpha\in\A$ is called a loop (resp.\ coloop) 
if $\alpha\in\Gamma_{\tor}$ (resp.\ $r_\A=r_{\A\smallsetminus\{\alpha\}}+1$). 
An element $\alpha$ that is neither a loop nor a coloop is called proper 
(\cite[\S 4.4]{dad-mo}). 

\begin{lemma}
\label{lem:del-cont}
Let $(\A, \A', \A'')$ be the triple associated with 
$\alpha_i\in\A$. Then 
\begin{equation*}
Z_{\A}^G(q, \bm{v})=
\left\{
\begin{array}{ll}
Z_{\A'}^G(q, \bm{v})+
v_i\cdot Z_{\A''}^G(q, \bm{v}), & \mbox{ if $\alpha_i$ is a loop,} \\
Z_{\A'}^G(q, \bm{v})+
v_i\cdot q^{-1}\cdot Z_{\A''}^G(q, \bm{v}), & \mbox{ otherwise.} 
\end{array}
\right.
\end{equation*}
\end{lemma}
\begin{proof}
Similar to \cite[Lemma 3.2]{br-mo}. 
\end{proof}

\begin{corollary}
\label{cor:del-cont1}
The $G$-Tutte polynomials satisfy 
\begin{equation*}
T_{\A}^G(x, y)=
\left\{
\begin{array}{ll}
T_{\A'}^G(x, y)+
(y-1) T_{\A''}^G(x, y), & \mbox{ if $\alpha_i$ is a loop,} \\
(x-1) T_{\A'}^G(x, y)+
T_{\A''}^G(x, y), & \mbox{ if $\alpha_i$ is a coloop,} \\
T_{\A'}^G(x, y)+
T_{\A''}^G(x, y),  & \mbox{ if $\alpha_i$ is proper.} 
\end{array}
\right.
\end{equation*}
\end{corollary}

\begin{corollary}
\label{cor:del-cont2}
The $G$-characteristic polynomials satisfy 
\begin{equation*}
\chi_{\A}^G(t)=
\chi_{\A'}^G(t)-
\chi_{\A''}^G(t). 
\end{equation*}
\end{corollary}

\subsection{Specializations}

\label{subsec:specializ}

The $G$-Tutte polynomial has several specializations. 

\begin{proposition}
Let $\A$ be a list in the free abelian group $\Gamma=\Z^\ell$. 
Suppose that $G$ is a torsion-free abelian group. Then 
$T_{\A}^G(x, y)=T_{\A}(x, y)$ and 
$\chi_{\A}^G(t)=\chi_{\A}(t)$. 
\end{proposition}
\begin{proof}
This follows from $\#\Hom((\Gamma/\langle\scS\rangle)_{\tor}, G)=1$ for 
$\scS\subset\A$. 
\end{proof}

\begin{proposition}
\label{prop:torarith}
Let $\A$ be a list in the finitely generated abelian 
group $\Gamma$, and let 
$G=S^1$ or $\C^\times$. 
Then $T_{\A}^G(x, y)=T_{\A}^{\arith}(x, y)$. 
\end{proposition}
\begin{proof}
Note that $\#\Hom((\Gamma/\langle\scS\rangle)_{\tor}, G)=
\#(\Gamma/\langle\scS\rangle)_{\tor}$, 
which is equal to the multiplicity 
$m(\scS)$ in the definition of the arithmetic Tutte polynomial. 
\end{proof}
The arithmetic Tutte polynomial can also be obtained as another 
specialization. 
Suppose that $(\Gamma/\langle\scS\rangle)_{\tor}
\simeq\bigoplus_{i=1}^{k_{\scS}}\Z/d_{\scS, i}\Z$, where 
$k_{\scS}\geq 0$ and $d_{\scS, i}\mid d_{\scS, i+1}$. 
Define $\rho_{\A}$ by 
\begin{equation*}
\rho_\A:=\lcm(d_{\scS, k_{\scS}}\mid\scS\subset\A). 
\end{equation*}

\begin{proposition}
\label{prop:rhoA}
$T_{\A}^{\Z/\rho_{\A}\Z}(x, y)=T_{\A}^{\arith}(x, y)$. 
\end{proposition}
\begin{proof}
Let $G=\Z/\rho_{\A}\Z$. 
By Proposition \ref{prop:finitetf}, 
because $d_{\scS, i}\mid\rho_{\A}$, 
we have 
$\#\Hom(\Z/d_{\scS, i}\Z, G)=d_{\scS, i}$ for all $\scS\subset\A$ and 
$1\leq i\leq k_{\scS}$. Furthermore, $m(\scS; G)=m(\scS)$ for all $\scS$. 
\end{proof}

\begin{example}
\label{ex:otherex1}
In \cite{car-spo, sokal}, partition functions of abelian group 
valued Potts models were studied, which were generalized to the arithmetic 
matroid setting by Br\"and\'en-Moci \cite[Theorem 7.4]{br-mo}. 
Br\"and\'en-Moci's polynomial $Z_{\mathcal{L}}(\Gamma, H, \bm{v})$ is, 
in our terminology, 
equal to
$(\#H)^{r_\Gamma}\cdot Z_{\mathcal{L}}^H(\#H, v_1, \dots, v_n)$, 
where 
$\mathcal{L}$ is a list of $n$ elements in $\Gamma$, and $H$ is a finite 
abelian group. In the same paper, Br\"and\'en-Moci also defined the Tutte 
quasi-polynomial $Q_{\mathcal{L}}(x, y)$, which is equal to 
$T_{\mathcal{L}}^{\Z/(x-1)(y-1)\Z}(x, y)$ 
for any fixed integers $x$ and $y$. 
\end{example}
\begin{example}
\label{ex:otherex2}
Both 
the modified Tutte-Krushkal-Renhardy polynomial for a finite CW-complex 
(see \cite[\S 3]{bbc}, \cite[\S 4]{de-mo} for details) and Bibby's 
Tutte polynomial for an elliptic arrangement (\cite[Remark 4.4]{bibby}) 
can be expressed using $T_{\A}^{S^1\times S^1}(x, y)$. 
\end{example}
We give a representation of the constituents of characteristic 
quasi-polynomials in terms of $G$-characteristic polynomials 
in \S \ref{subsec:count}.

\subsection{Changing the group $\Gamma$}

\label{subsec:changegamma}

Let $\sigma:\Gamma_1\longrightarrow\Gamma_2$ be a homomorphism between 
finitely generated abelian groups. 
The map $\sigma$ induces a homomorphism 
\begin{equation}
\label{eq:pullback}
\sigma^*:\Hom(\Gamma_2, G)\longrightarrow\Hom(\Gamma_1, G). 
\end{equation}
Let $\alpha\in\Gamma_1$. It is easily seen that 
$(\sigma^*)^{-1}(H_{\alpha, G})=H_{\sigma(\alpha), G}$. Hence 
\eqref{eq:pullback} induces a map between the complements 
\begin{equation*}
\sigma^*|_{\M(\sigma(\A); \Gamma_2, G)}:
\M(\sigma(\A); \Gamma_2, G)\longrightarrow\M(\A; \Gamma_1, G). 
\end{equation*}

A natural question is to compare 
$T_{\A}^G(x, y)$ and $T_{\sigma(\A)}^G(x, y)$. 
This comparison is in general difficult. 
However, in the case where 
$\Gamma_1=\Gamma_2=\Z^\ell$ and $G$ is a connected Lie group, 
the constant terms of the $G$-characteristic polynomials can be controlled 
by $\det(\sigma)$. 

\begin{proposition}
\label{prop:latticechange}
Let $\Gamma=\Z^\ell$, 
$\sigma:\Gamma\longrightarrow\Gamma$ be a homomorphism, 
$\A$ be a finite list of elements in $\Gamma$, and 
$G=(S^1)^p\times\R^q$ with $p>0$. 
Then 
\begin{equation*}
\chi_{\sigma(\A)}^{G}(0)=|\det(\sigma)|^p\cdot\chi_{\A}^{G}(0). 
\end{equation*}
\end{proposition}
\begin{proof}
By the definition (Definition \ref{def:main}) 
of the $G$-characteristic polynomial, 
$\chi_{\A}^G(t)$ is divisible by $t^{r_{\Gamma}-r_{\A}}$. 
If $\det(\sigma)=0$, then $r_{\sigma(\A)}<\ell=r_{\Gamma}$, and 
$\chi_{\sigma(\A)}^G(t)$ is divisible by $t$. 
Therefore the left-hand side vanishes, and the assertion holds trivially. 

We assume instead that $\det(\sigma)\neq 0$. 
Note that for a sublattice $L\subset\Gamma$ of rank $\ell$, 
we have $(\Gamma:\sigma(L))=|\det(\sigma)|\cdot (\Gamma:L)$. 
Second, if $r_{\scS}=\ell$, then 
$(\Gamma/\langle\scS\rangle)_{\tor}=\Gamma/\langle\scS\rangle$, and 
we have 
$m(\scS; G)=m(\scS; S^1)^p=\#(\Gamma/\langle\scS\rangle)^p$. 
Third, because $\sigma:\Gamma\longrightarrow\Gamma$ is injective, 
$r_{\sigma(\scS)}=r_{\scS}$ and $\#\sigma(\scS)=\#\scS$ 
for every sublist $\scS\subset\A$. 
Therefore, 
\begin{equation*}
\begin{split}
\chi_{\sigma(\A)}^{G}(0)
&=
\sum_{\substack{\sigma(\scS)\subset\sigma(\A) \\ r_{\sigma(\scS)}=\ell}}(-1)^{\#\sigma(\scS)}m(\sigma(\scS); G)\\
&=
\sum_{\substack{\scS\subset\A \\ r_{\scS}=\ell}}(-1)^{\#\scS}m(\sigma(\scS); G)\\
&=
\sum_{\substack{\scS\subset\A \\ r_{\scS}=\ell}}(-1)^{\#\scS}|\det(\sigma)|^p m(\scS; G)\\
&=
|\det(\sigma)|^p\cdot\chi_{\A}^{G}(0).
\qedhere
\end{split}
\end{equation*}
\end{proof}

\subsection{The case $\Gamma$ is finite}

\label{subsec:finiteGamma}

If the group $\Gamma$ is finite (or equivalently $r_{\Gamma}=0$), then 
$r_{\A}=r_{\scS}=0$. Hence $T_{\A}^G(x, y)$ is a polynomial in $y$ 
by definition. In this case, 
the coefficients of the $G$-Tutte polynomial can be explicitly expressed. 
More generally, we can prove the following. 

\begin{theorem}
\label{thm:finpos}
Let $\A$ be a finite list of elements in a 
finitely generated abelian group $\Gamma$, and let $G$ be a 
torsion-wise finite abelian group. 
Suppose that $\A$ is contained in $\Gamma_{\tor}$. Then 
\begin{equation*}
T_{\A}^{G}(x, y)=
\sum_{\underset{\bullet}{k}=0}^{\#\A}
\left(
\sum_{\substack{\underset{\bullet}{\scS}\subset\A\\ \#\underset{\bullet}{\scS}=k}}
\#\M(\A/\scS; \Gamma_{\tor}/\langle\scS\rangle, G)
\right)
y^k. 
\end{equation*}
In particular, $T_{\A}^G(x, y)$ is a polynomial in $y$ with 
positive coefficients. 
\end{theorem}
\begin{proof}
By assumption, $r_{\A}=r_{\scS}=0$ and 
$(\Gamma/\langle\scS\rangle)_{\tor}=
\Gamma_{\tor}/\langle\scS\rangle$ for every $\scS\subset\A$. 
Using Proposition \ref{prop:disjoint}, we have 
\begin{equation*}
{\small 
\begin{split}
T_{\A}^{G}(x, y)
=&
\sum_{\underset{\bullet}{\scS}\subset\A}
\#\Hom(\Gamma_{\tor}/\langle\scS\rangle, G)\cdot(y-1)^{\#\scS}
\\
=&
\sum_{\underset{\bullet}{\scS}\subset\A}
\#\Hom(\Gamma_{\tor}/\langle\scS\rangle, G)\cdot
\sum_{\underset{\bullet}{k}=0}^{\#\scS}y^k\cdot (-1)^{\#\scS-k}\cdot
\begin{pmatrix}
\#\scS\\ k
\end{pmatrix}
\\
=&
\sum_{\underset{\bullet}{k}=0}^{\#\A}y^k\cdot
\sum_{\substack{\underset{\bullet}{\scS}\subset\A\\ \#\underset{\bullet}{\scS}\geq k}}
(-1)^{\#\scS-k}
\begin{pmatrix}
\#\scS\\ k
\end{pmatrix}
\sum_{\scS\subset\underset{\bullet}{\scT}\subset\A}\#\M(\A/\scT; \Gamma_{\tor}/\langle\scT\rangle, G)
\\
=&
\sum_{\underset{\bullet}{k}=0}^{\#\A}y^k\cdot
\sum_{\substack{\underset{\bullet}{\scT}\subset\A \\ \#\underset{\bullet}{\scT}\geq k}}\#\M(\A/\scT; \Gamma_{\tor}/\langle\scT\rangle, G)\cdot
\sum_{\substack{\underset{\bullet}{\scS}\subset\scT\\ \#\underset{\bullet}{\scS}\geq k}}
(-1)^{\#\scS-k}
\begin{pmatrix}
\#\scS\\ k
\end{pmatrix}
\\
=&
\sum_{\underset{\bullet}{k}=0}^{\#\A}y^k\cdot
\sum_{\substack{\underset{\bullet}{\scT}\subset\A \\ \#\underset{\bullet}{\scT}\geq k}}\#\M(\A/\scT; \Gamma_{\tor}/\langle\scT\rangle, G)\cdot
\sum_{k\leq\underset{\bullet}{m}\leq\#\scT}(-1)^{m-k}
\begin{pmatrix}
m\\ k
\end{pmatrix}
\cdot
\begin{pmatrix}
\#\scT\\ m
\end{pmatrix}
\\
=&
\sum_{\underset{\bullet}{k}=0}^{\#\A}y^k\cdot
\sum_{\substack{\underset{\bullet}{\scT}\subset\A \\ \#\underset{\bullet}{\scT}\geq k}}\#\M(\A/\scT; \Gamma_{\tor}/\langle\scT\rangle, G)\cdot
\begin{pmatrix}
\#\scT\\ k
\end{pmatrix}
\sum_{k\leq\underset{\bullet}{m}\leq\#\scT}(-1)^{m-k}
\begin{pmatrix}
\#\scT-k\\ m-k
\end{pmatrix}
\\
=&
\sum_{\underset{\bullet}{k}=0}^{\#\A}y^k\cdot
\sum_{\substack{\underset{\bullet}{\scT}\subset\A \\ 
\#\underset{\bullet}{\scT}= k}} \#\M(\A/\scT; \Gamma_{\tor}/\langle\scT\rangle, G).
\end{split}
}
\qedhere
\end{equation*}
\end{proof}

\section{Euler characteristic and point counting}

\label{sec:eulercount}

In this section, we prove formulas that express the 
Euler characteristic (in the case that $G$ is an abelian Lie group 
with finitely many components) and the cardinality 
(in the case that $G$ is finite) of the complement as 
a special value of the $G$-characteristic polynomial. 
We then describe the constituents of characteristic quasi-polynomials 
in terms of $G$-characteristic polynomials. 

\subsection{Euler characteristic of the complement}

\label{subsec:eulerchar}

We first recall the notion of Euler characteristic for semialgebraic sets 
(see \cite{cos-ras, bpr} for further details). 
Every semialgebraic set $X$ has a decomposition 
$X=\bigsqcup_{i=1}^N X_i$ such that each $X_i$ is a semialgebraic 
subset that is semialgebraically homeomorphic to 
the open simplex $\sigma_{d_i}=\{(x_1, \dots, x_{d_i})\in\R^{d_i}
\mid x_i>0, \sum x_i<1\}$ for some $d_i=\dim X_i$. 
The semialgebraic 
Euler characteristic of $X$ is defined by 
\begin{equation*}
e_{\semi}(X)=\sum_{i=1}^N(-1)^{d_i}. 
\end{equation*}
The Euler characteristic $e_{\semi}(X)$ is independent of the 
choice of decomposition. Furthermore, it satisfies the following 
additivity and multiplicativity properties: 
\begin{itemize}
\item 
Let $X$ be a semialgebraic set. Let $Y\subset X$ be a semialgebraic 
subset. Then $e_{\semi}(X)=e_{\semi}(X\smallsetminus Y)+e_{\semi}(Y)$. 
\item 
Let $X$ and $Y$ be semialgebraic sets. Then $e_{\semi}(X\times Y)=
e_{\semi}(X)\times e_{\semi}(Y)$. 
\end{itemize}
\begin{remark}
Unlike the topological Euler characteristic $e_{\top}(X)=\sum(-1)^ib_i(X)$, 
the semialgebraic Euler characteristic 
$e_{\semi}(X)$ is not homotopy invariant. (Even contractible 
semialgebraic sets have different values: e.g., 
$e_{\semi}([0, 1])=1$, $e_{\semi}(\R_{\geq 0})=0$, $e_{\semi}(\R)=-1$.) 
However, if $X$ is a manifold (without boundary), then 
$e_{\semi}(X)$ and $e_{\top}(X)$ are related by the following formula: 
\begin{equation*}
e_{\semi}(X)=(-1)^{\dim X}\cdot e_{\top}(X). 
\end{equation*}
\end{remark}

Here we assume that $G$ is an abelian Lie group with 
finitely many connected components. Then $G$ is of the form 
$G=(S^1)^{p}\times \R^{q}\times F$, 
where $F$ is a finite abelian group. 
Such a group $G$ can be realized as a semialgebraic set, with 
the group operations defined by $C^{\infty}$ semialgebraic maps. 
Hence subsets defined by using group operations are always semialgebraic 
sets. 

The Euler characteristics of $G$ are easily computed as 
\begin{equation*}
\begin{split}
&
e_{\semi}(G)=
\left\{
\begin{array}{cc}
0, & \mbox{if $p>0$, }\\
(-1)^{p+q}\cdot\# F,& \mbox{if $p=0$, }
\end{array}
\right.
\\
& 
e_{\top}(G)=
\left\{
\begin{array}{cc}
0, & \mbox{if $p>0$, }\\
\# F,& \mbox{if $p=0$. }
\end{array}
\right.
\end{split}
\end{equation*}
Let $\A$ be a finite list of elements in a finitely generated abelian group 
$\Gamma$. 
The space $\M(\A; \Gamma, G)$ is a semialgebraic set, and, if it is not empty, 
it is also a manifold (without boundary) of 
$\dim \M(\A; \Gamma, G)=r_\Gamma\cdot\dim G$. 

The $G$-Tutte polynomial can be used to compute the Euler characteristic of 
$\M(\A; \Gamma, G)$. 

\begin{theorem}
\label{thm:euler}
Let $G$ be an abelian Lie group with 
finitely many connected components, and let $g=\dim G$. 
Then, 
\begin{equation*}
e_{\semi}(\M(\A; \Gamma, G))
=\chi_{\A}^{G}(e_{\semi}(G)), 
\end{equation*}
or equivalently, 
\begin{equation*}
e_{\top}(\M(\A; \Gamma, G))
=(-1)^{g\cdot r_\Gamma} 
\cdot \chi_{\A}^{G}\left((-1)^g\cdot e_{\top}(G)\right). 
\end{equation*}
\end{theorem}

\begin{proof}
By the additivity of $e_{\semi}(-)$, we can compute 
$e_{\semi}(\M(\A; \Gamma, G))$ by using the principle of inclusion-exclusion 
together with Proposition \ref{prop:intersect} as follows: 
\begin{equation*}
\begin{split}
e_{\semi}(\M(\A; \Gamma, G))
&=
\sum_{\scS\subset\A}(-1)^{\#\scS}\cdot 
e_{\semi}\left(\bigcap_{\alpha\in\scS}H_{\alpha, G}\right)\\
&=
\sum_{\scS\subset\A}(-1)^{\#\scS}\cdot m(\scS; G)\cdot 
e_{\semi}(G)^{r_{\Gamma}-r_\scS}
\\
&=
\chi_{\A}^{G}(e_{\semi}(G)). 
\end{split}
\end{equation*}
\end{proof}

\begin{remark}
We can also prove Theorem \ref{thm:euler} by using 
deletion-contraction formula. 
Note that if $\A=\emptyset$, then 
$\chi_{\A}^G(t)=
\#\Hom(\Gamma_{\tor}, G)\cdot t^{r_\Gamma}$. Hence 
$\chi_{\A}^G(e_{\semi}(G))=
\#\Hom(\Gamma_{\tor}, G)\cdot e_{\semi}(G)^{r_\Gamma}=
e_{\semi}(\Hom(\Gamma, G))$. 
Theorem \ref{thm:euler} then follows easily from Proposition \ref{prop:set} 
and Corollary \ref{cor:del-cont2} by induction on $\#\A$. 
\end{remark}

\subsection{Point counting in complements}

\label{subsec:count}

In the case that $G$ is finite, the complement $\M(\A; \Gamma, G)$ is also a finite set. 
Every finite set can be considered as a $0$-dimensional 
semialgebraic set whose Euler characteristic is equal to its cardinality. 
The following theorem immediately follows from Theorem \ref{thm:euler}. 

\begin{theorem}
\label{thm:count}
Let $\A$ be a finite list of elements in a finitely generated 
abelian group $\Gamma$, 
and let $G$ be a finite abelian group. Then 
\begin{equation*}
\#\M(\A; \Gamma, G)=
\chi_{\A}^G(\#G). 
\end{equation*}
\end{theorem}

We can now describe the constituents of characteristic quasi-polynomials 
as $G$-characteristic polynomials (see \S \ref{subsec:cqp}). 
\begin{theorem}
\label{thm:cqp}
(See \S\ref{subsec:cqp} for notation.) 
Let $\A$ be a finite list of elements in $\Gamma=\Z^\ell$, 
and let $k$ be a divisor of $\rho_{\A}$. 
The $k$-constituent $f_k(t)$ of the characteristic quasi-polynomial 
$\chi^{\quasi}_{\A}(q)$ is equal to 
\begin{equation*}
f_k(t)=
\chi_{\A}^{\Z/k\Z}(t). 
\end{equation*}
\end{theorem}
\begin{proof}
Let $q\in\Z_{>0}$ be a positive integer,  
and suppose that $\gcd(q, \rho_\A)=k$. 
Because $d_{\scS, i}\mid\rho_{\A}$, we have 
\begin{equation*}
\gcd(q, d_{\scS, i})=
\gcd(k, d_{\scS, i})
\end{equation*}
for any $\scS\subset\A$ and $1\leq i\leq r_\scS$. 
It follows from Proposition \ref{prop:finitetf} that 
$\#(\Z/q\Z)[d_{\scS, i}]=\#(\Z/k\Z)[d_{\scS, i}]$, and hence 
\begin{equation*}
m(\scS; \Z/q\Z)=m(\scS; \Z/k\Z). 
\end{equation*}
Using Theorem \ref{thm:count}, 
\begin{equation*}
\begin{split}
f_k(q)&=\#\M(\A; \Gamma, \Z/q\Z)=\chi_{\A}^{\Z/q\Z}(q)\\
&=
\sum_{\scS\subset\A}(-1)^{\#\scS} m(\scS; \Z/q\Z)\cdot q^{r_{\Gamma}-r_{\scS}}
\\
&=
\sum_{\scS\subset\A}(-1)^{\#\scS} m(\scS; \Z/k\Z)\cdot q^{r_{\Gamma}-r_{\scS}}
\\
&=
\chi_{\A}^{\Z/k\Z}(q). 
\end{split}
\end{equation*}
Because $f_k(t)$ and $\chi_{\A}^{\Z/k\Z}(t)$ are 
polynomials in $t$ that have common values for infinitely 
many $q>0$, $f_k(t)=\chi_{\A}^{\Z/k\Z}(t)$. 
\end{proof}

\begin{corollary}
\label{cor:toricchar}
The most degenerate constituent $f_{\rho_\A}(t)$ of the characteristic 
quasi-polynomial $\chi^{\quasi}_{\A}(q)$ is equal to both 
$\chi_{\A}^{\C^\times}(t)$ and $\chi_{\A}^{\arith}(t)$. 
\end{corollary}
\begin{proof}
By Proposition \ref{prop:torarith}, 
Proposition \ref{prop:rhoA} (and its specialization to 
the characteristic polynomial) and Theorem \ref{thm:cqp}, we have 
\begin{equation*}
\chi_{\A}^{\arith}(t)=
\chi_{\A}^{\C^\times}(t)=
\chi_{\A}^{\Z/\rho_\A\Z}(t)=
f_{\rho_\A}(t). 
\end{equation*}
\end{proof}

\begin{remark}
\label{rem:ffm}
Corollary \ref{cor:toricchar} enables us to compute 
the Poincar\'e polynomial 
$P_{\M(\A; \Gamma, \C^\times)}(t)$ of the associated toric arrangement 
$\A(\C^\times)$ via modulo $q$ counting, where $\rho_{\A}\mid q$. 

In particular, the Poincar\'e polynomial of a toric arrangement 
can be computed if the characteristic quasi-polynomial 
is known. We will compute the Poincar\'e polynomial of the toric 
arrangement for exceptional root systems in \S \ref{sec:example}. 
\end{remark}

\section{Examples: root systems}

\label{sec:example}

As we saw in \S \ref{subsec:count} (Remark \ref{rem:ffm}), 
the Poincar\'e polynomial of a toric arrangement can be computed 
from the characteristic quasi-polynomial. 
By applying recent results on characteristic quasi-polynomials 
of root systems, we prove that Poincar\'e polynomials 
satisfy a certain functional equation. We also prove a formula 
that expresses the Euler characteristic. 

Let $\Phi$ be an irreducible root system of rank $\ell$, and 
let $\Gamma=\Z\cdot\Phi$ be 
the root lattice of $\Phi$. Consider the list $\A_{\Phi}:=\Phi^+\subset\Gamma$ 
of positive roots. 
The characteristic quasi-polynomial $\chi_{\A_\Phi}^{\quasi}(q)$ 
was computed by Suter \cite{sut} and 
Kamiya-Takemura-Terao \cite{ktt-quasi}. 
The most degenerate constituents $f_{\rho_{\A_\Phi}}(t)$ 
are shown in Table \ref{fig:table}. 

\begin{table}[htbp]
\centering
{\footnotesize 
\begin{tabular}{c|c|c|c|c|l}
$\Phi$&$\#W$&$h$&$f$&$\rho_{\A_{\Phi}}$&$f_{\rho_{\A}}(t)=\chi_{\A}^{\C^\times}(t)$\\
\hline\hline
$A_\ell$&$(\ell+1)!$&$\ell+1$&$\ell+1$&$1$&$\displaystyle\prod_{k=1}^\ell(t-k)$\\
\hline
$B_\ell, C_\ell$&$2^\ell\cdot\ell!$&$2\ell$&$2$&$2$&$(t-\ell)\displaystyle\prod_{k=1}^{\ell-1}(t-2k)$\\
\hline
$D_\ell$&$2^{\ell-1}\cdot\ell!$&$2\ell-2$&$4$&$2$&$(t^2-2(\ell-1)t+\frac{\ell(\ell-1)}{2})\displaystyle\prod_{k=1}^{\ell-2}(t-2k)$\\
\hline
$E_6$&$2^7\cdot 3^4\cdot 5$&$12$&$3$&$6$&$(t - 6)^2 (t^4 - 24 t^3 + 186 t^2 - 504 t + 480)$\\
\hline
$E_7$&$2^{10}\cdot 3^4\cdot 5\cdot 7$&$18$&$2$&$12$&$(t - 12) (t^6 - 51 t^5 + 1005 t^4 - 9675 t^3 $\\
&&&&&$+ 47784 t^2 - 116064 t + 120960)$\\
\hline
$E_8$&$2^{14}\cdot 3^5\cdot 5^2\cdot 7$&$30$&$1$&$60$&$t^8 - 120 t^7 + 6020 t^6 - 163800 t^5 + 2626008 t^4$\\
&&&&&$ - 25260480 t^3 +  142577280 t^2 - 445824000 t $\\
&&&&&$ + 696729600$\\
\hline
$F_4$&$2^7\cdot 3^2$&$12$&$1$&$12$&$t^4 - 24 t^3 + 208 t^2 - 768 t + 1152$\\
\hline
$G_2$&$2^2\cdot 3$&$6$&$1$&$6$&$t^2-6 t+12$
\end{tabular}
}
\caption{Table of root systems. (Notation: $W$ is the Weyl group, $h$ is the Coxeter number, $f$ is the index of connection, and $\rho_{\A_\Phi}$ is the minimal period of the characteristic quasi-polynomial. See \cite{hum} for details.) }
\label{fig:table}
\end{table}

Using formula \eqref{eq:torpoin}, or Theorem \ref{thm:poin}, 
the Poincar\'e polynomial for the corresponding toric arrangement is 
$P_{\M(\A_{\Phi}; \Gamma, \C^\times)}(t)=(-t)^\ell\chi_{\A_\Phi}^{\C^\times}
\left(-\frac{1+t}{t}\right)$. 
We only show exceptional cases. 
(See \cite{bergv} for more detailed information about the cohomology groups 
including the Weyl group action (except for $E_8$).) 
\begin{equation*}
\begin{split}
P_{\M(\A_{E_6}; \Gamma, \C^\times)}(t)=&1 + 42 t + 705 t^2 + 6020 t^3 + 27459 t^4 + 63378 t^5 + 58555 t^6,\\
P_{\M(\A_{E_7}; \Gamma, \C^\times)}(t)=&1 + 70 t + 2016 t^2 + 30800 t^3 + 268289 t^4 + 1328670 t^5 \\
& + 3479734 t^6 + 3842020 t^7,\\
P_{\M(\A_{E_8}; \Gamma, \C^\times)}(t)=&1 + 128 t + 6888 t^2 + 202496 t^3 + 3539578 t^4 + 37527168 t^5 \\
&+  235845616 t^6 + 818120000 t^7 + 1313187309 t^8,\\
P_{\M(\A_{F_4}; \Gamma, \C^\times)}(t)=&1 + 28 t + 286 t^2 + 1260 t^3 + 2153 t^4,\\
P_{\M(\A_{G_2}; \Gamma, \C^\times)}(t)=&1 + 8 t + 19 t^2.
\end{split}
\end{equation*}

It was proved in \cite[Corollary 3.8]{yos-wor} that the 
characteristic quasi-polynomial of a root system $\Phi$ satisfies 
the functional equation 
\begin{equation*}
\chi_{\A_{\Phi}}^{\quasi}(h-q)=(-1)^\ell \chi_{\A_{\Phi}}^{\quasi}(q), 
\end{equation*}
where $h$ is the Coxeter number. 
This equation holds even at the level of $k$-constituents. 
\begin{equation}
\label{eq:constfe}
f_k(h-q)=(-1)^\ell f_k(q), 
\end{equation}
for some $k$. 
(More precisely, \eqref{eq:constfe} holds for admissible divisors $k$ 
in the sense of \cite[\S 5.3]{yos-linial}.) 
In particular, equation \eqref{eq:constfe} holds for 
$k=\rho_{\A_\Phi}$ if and only if the period 
$\rho_{\A_\Phi}$ divides the Coxeter number $h$, 
which is equivalent to $\Phi\neq E_7, E_8$ 
(see \cite[\S 5.3]{yos-linial} for details). Thus we have the following. 

\begin{proposition}
\label{prop:functeq}
Let $\Phi$ be an irreducible root system. Assume that $\Phi\neq E_7, E_8$. 
Then 
the characteristic polynomial of the associated toric arrangement 
satisfies 
\begin{equation*}
\chi_{\A_\Phi}^{\C^\times}(h-t)=(-1)^\ell
\chi_{\A_\Phi}^{\C^\times}(t), 
\end{equation*}
or, equivalently, the Poincar\'e polynomial satisfies the following relation: 
\begin{equation*}
P_{\M(\A_\Phi; \Gamma, \C^\times)}(t)=
\left( (h+2)t+1\right)^\ell \cdot 
P_{\M(\A_\Phi; \Gamma, \C^\times)}\left(
\frac{-t}{(h+2)t+1}
\right). 
\end{equation*}
\end{proposition}

We next describe the Euler characteristic of $\M(\A_{\Phi}; \Gamma, \C^\times)$. 

\begin{proposition} 
\label{prop:constantterm}

Let $W$ be the Weyl group of $\Phi$, and let $f$ be the index of connection. The constant term of the characteristic polynomial of the associated toric arrangement can be computed as follows:
\begin{equation*}
\chi_{\A_\Phi}^{\C^\times}(0) =\frac{(-1)^\ell\#W}{f}.
\end{equation*}
\end{proposition}
\begin{proof}
By  Corollary \ref{cor:toricchar},
\begin{equation*}
\chi_{\A_\Phi}^{\C^\times}(0)=
f_{\rho_{\A_\Phi}}(0)=\frac{(-1)^\ell\#W}{f}L_{\overline{A^\circ}}(0)=\frac{(-1)^\ell\#W}{f},
\end{equation*}
 where $L_{\overline{A^\circ}}(q)$ is the Ehrhart quasi-polynomial of the closed fundamental alcove $\overline{A^\circ}$ (see \cite{hum, yos-wor} for the definition of $\overline{A^\circ}$). 
The second equality is obtained from \cite[Proposition 3.7]{yos-wor}, and the last equality follows from $L_{\overline{A^\circ}}(0)=1$. 
\end{proof}

\begin{remark}
\label{rem:Cartan}
The Cartan matrix of $\Phi$ whose determinant is the index of connection $f$ expresses the change of basis between the root lattice and the weight lattice. 
It follows from Propositions \ref{prop:latticechange} and \ref{prop:constantterm} that the constant term of the characteristic polynomial of the toric arrangement with respect to the weight lattice equals $(-1)^\ell\#W$. 
This gives a new proof for \cite[Corollary 7.4]{moci-tor}.
\end{remark}

Using the notation in Section \ref{sec:eulercount}, we can compute the Euler characteristic of 
$\M(\A_\Phi; \Gamma, \C^\times)$ as follows, noting that $e_{\top}(\C^\times)=
e_{\semi}(\C^\times)=0$:
\begin{corollary}
\label{cor:toriccharrootsystem}
(\cite{moci-lincei, moci-tor})
\begin{equation*}
e_{\semi}(\M(\A_\Phi; \Gamma, \C^\times))= e_{\top}(\M(\A_\Phi; \Gamma, \C^\times))
=\frac{(-1)^\ell\#W}{f}.
\end{equation*}
\end{corollary}
\begin{proof}
This follows directly from Theorem \ref{thm:euler} and Proposition \ref{prop:constantterm}.
\end{proof}

\section{Poincar\'e polynomials for non-compact groups}

\label{sec:poinc}

In this section, we prove a formula that expresses 
the Poincar\'e polynomial in terms of the $G$-characteristic polynomial. 

\subsection{Torus cycles}
\label{subsec:tcycle}

We introduce a special class of homology cycles called torus cycles 
in $H_{*}(\M(\A; \Gamma, G), \Z)$, which are lifts of cycles in a compact torus. 

Let $G=F\times (S^1)^p\times\R^q$, where $F$ a finite abelian group. 
Write $G_{\operatorname{c}}=F\times (S^1)^p$ (compact part) and $V=\R^q$ (non-compact part). 

Let $\Gamma$ be a finitely generated abelian group. Fix a decomposition 
$\Gamma=\Gamma_{\tor}\oplus\Gamma_{\free}$, 
where $\Gamma_{\free}\simeq\Z^{r_\Gamma}$.  
Then 
\begin{equation}
\label{eq:decomp1}
\Hom(\Gamma, G)
\simeq 
\Hom(\Gamma, G_{\operatorname{c}})\times\Hom(\Gamma_{\free}, V). 
\end{equation}
(Note that $\Hom(\Gamma_{\tor}, V)=0$.) 
We can decompose this further as follows: 
\begin{equation}
\label{eq:decomp2}
\Hom(\Gamma, G)
\simeq 
\Hom(\Gamma_{\tor}, G_{\operatorname{c}})\times\Hom(\Gamma_{\free}, G_{\operatorname{c}})\times\Hom(\Gamma_{\free}, V). 
\end{equation}

The first component $\Hom(\Gamma_{\tor}, G_{\operatorname{c}})$ of \eqref{eq:decomp2} 
is a finite abelian group, 
the second component $\Hom(\Gamma_{\free}, G_{\operatorname{c}})$ is a compact abelian Lie group 
(not necessarily connected), and the third component is 
$\Hom(\Gamma_{\free}, V)\simeq V^{r_{\Gamma}}\simeq\R^{q\cdot r_{\Gamma}}$. 

Let 
$\alpha=(\beta, \eta)\in \Gamma_{\tor}\oplus\Gamma_{\free}$. 
According to decomposition \eqref{eq:decomp1}, the subgroup 
$H_{\alpha, G}\subset\Hom(\Gamma, G)$ can be expressed as 
\begin{equation*}
H_{\alpha, G}=
H_{\alpha, G_{\operatorname{c}}}\times
H_{\eta, V}, 
\end{equation*}
where $H_{\alpha, G_{\operatorname{c}}}\subset\Hom(\Gamma, G_{\operatorname{c}})$ and 
$H_{\eta, V}\subset\Hom(\Gamma_{\free}, V)$. 

If $\alpha\in\Gamma_{\tor}$, or equivalently $\alpha=(\beta, 0)$, 
then using \eqref{eq:decomp2} gives 
\begin{equation*}
H_{\alpha, G}=
H_{\beta, G_{\operatorname{c}}}\times
\Hom(\Gamma_{\free}, G_{\operatorname{c}})
\times
\Hom(\Gamma_{\free}, V), 
\end{equation*}
where $H_{\beta, G_{\operatorname{c}}}$ is a subgroup of the finite abelian group 
$\Hom(\Gamma_{\tor}, G_{\operatorname{c}})$. 
In this case, $H_{\alpha, G}$ is a collection of connected components 
of $\Hom(\Gamma, G)$ (note that $H_{\beta, G_{\operatorname{c}}}$ is a finite subset of 
the finite abelian group $\Hom(\Gamma_{\tor}, G_{\operatorname{c}})$). 
Similarly, the complement can be expressed as 
\begin{equation*}
\begin{split}
\M(\{\alpha\}; \Gamma, G)&=\Hom(\Gamma, G)\smallsetminus H_{\alpha, G}\\
&=
\left(\Hom(\Gamma_{\tor}, G_{\operatorname{c}})\smallsetminus H_{\beta, G_{\operatorname{c}}}\right)
\times
\Hom(\Gamma_{\free}, G_{\operatorname{c}})
\times
\Hom(\Gamma_{\free}, V). 
\end{split}
\end{equation*}
More generally, if $\A\subset\Gamma_{\tor}\subset\Gamma$, then 
\begin{equation}
\label{eq:torconn}
\begin{split}
\M(\A; \Gamma, G)
&=
\left(
\Hom(\Gamma_{\tor}, G_{\operatorname{c}})\smallsetminus
\bigcup_{\alpha\in\A}H_{\alpha, G_{\operatorname{c}}}
\right)
\times
\Hom(\Gamma_{\free}, G_{\operatorname{c}})
\times
\Hom(\Gamma_{\free}, V)\\
&=
\M(\A; \Gamma_{\tor}, G_{\operatorname{c}})\times
\Hom(\Gamma_{\free}, G_{\operatorname{c}})
\times
\Hom(\Gamma_{\free}, V). 
\end{split}
\end{equation}
Therefore, $\M(\A; \Gamma, G)$ is a collection of some of connected components 
of $\Hom(\Gamma, G)$. 

Let $\A\subset\Gamma$ be a list of elements. 
Define $\A_{\tor}:=\A\cap\Gamma_{\tor}$. As mentioned above, 
$\M(\A_{\tor}; \Gamma, G)$ is a collection of components of 
$\Hom(\Gamma, G)$. 

Consider the following diagram: 
\begin{equation}
\label{eq:diagramprojection}
{\small 
\begin{CD}
\M(\A; \Gamma, G) @> \subset >> \M(\A_{\tor}; \Gamma, G) @> \subset >> \Hom(\Gamma, G) @. \ni (f, t, v)\\
@. @VVV @VV\pi V\\
@. \M(\A_{\tor}; \Gamma, G_{\operatorname{c}}) @> \subset >> \Hom(\Gamma, G_{\operatorname{c}}) @. \ni (f, t), 
\end{CD}
}
\end{equation}
where $\pi:\Hom(\Gamma, G)\longrightarrow\Hom(\Gamma, G_{\operatorname{c}})$ is the projection 
defined by $\pi(f, t, v)=(f, t)$ for 
$(f, t, v)\in 
\Hom(\Gamma_{\tor}, G_{\operatorname{c}})\times\Hom(\Gamma_{\free}, G_{\operatorname{c}})\times
\Hom(\Gamma_{\free}, V)\simeq\Hom(\Gamma, G)$.

Now assume that $q>0$. The fiber of the projection $\pi$ is 
isomorphic to $\Hom(\Gamma, V)\simeq 
V^{r_\Gamma}\simeq \R^{q\cdot r_\Gamma}$. 
Then 
\begin{equation*}
\M(\A\smallsetminus\A_{\tor}; \Gamma, V)=\Hom(\Gamma, V)\smallsetminus
\bigcup_{\alpha\in\A\smallsetminus\A_{\tor}}H_{\alpha, V}
\end{equation*}
is a complement of proper subspaces. Hence it 
is non-empty. 
Fix an element 
$v_0\in \M(\A\smallsetminus\A_{\tor}; \Gamma, V)$. 
For a given $(f, t)\in \Hom(\Gamma, G_{\operatorname{c}})$, 
define $i_{v_0}(f, t):=(f, t, v_0)$. This induces a map 
\begin{equation*}
i_{v_0}: 
\M(\A_{\tor}; \Gamma, G_{\operatorname{c}})\longrightarrow \M(\A; \Gamma, G), 
\end{equation*}
which is a section of the projection 
$\pi|_{\M(\A; \Gamma, G)}:\M(\A; \Gamma, G)\longrightarrow\M(\A_{\tor}; \Gamma, G_{\operatorname{c}})$ 
in \eqref{eq:diagramprojection}. 

\begin{definition}
Assume that $q>0$. 
A cycle $\gamma\in H_*(\M(\A; \Gamma, G), \Z)$ is said to be a \emph{torus cycle} 
if there exist a connected component 
$T\subset\M(\A_{\tor}; \Gamma, G_{\operatorname{c}})$, a cycle 
$\widetilde{\gamma}\in H_*(T, \Z)\subset H_*(\M(\A_{\tor}; \Gamma, G_{\operatorname{c}}), \Z)$ and 
$v_0\in \M(\A\smallsetminus\A_{\tor}; \Gamma, V)$ such that 
\begin{equation*}
\gamma=(i_{v_0})_*(\widetilde{\gamma}). 
\end{equation*}
The subgroup of $H_*(\M(\A; \Gamma, G), \Z)$ generated by 
torus cycles is denoted by $H_*^{\torus}(\A(G))$. 
\end{definition}

\begin{remark}
\label{rem:dependence}
If $q>1$, then the homology class 
$(i_{v_0})_*(\widetilde{\gamma})$ is independent 
of the choice of ${v_0}\in \M(\A\smallsetminus\A_{\tor}; \Gamma, V)$, 
because 
$\M(\A\smallsetminus\A_{\tor}; \Gamma, V)$ is connected. 
On the other hand, if $q=1$, then the subspace $H_{\alpha, V}$ is 
a real hyperplane in $\Hom(\Gamma, V)\simeq V^{r_\Gamma}$. 
Hence the homology class 
$(i_{v_0})_*(\widetilde{\gamma})$ may depend 
on the chamber containing ${v_0}$. 
\end{remark}

\begin{lemma}
\label{lem:surjtorus}
Assume that $q>0$. 
Let $\alpha\in\A\smallsetminus\A_{\tor}$, and 
$\A'=\A\smallsetminus\{\alpha\}$. Then the map 
$\iota: H_*^{\torus}(\A(G))\longrightarrow H_*^{\torus}(\A'(G))$ 
induced by the inclusion $\M(\A; \Gamma, G)\hookrightarrow\M(\A'; \Gamma, G)$ 
is surjective. 
\end{lemma}

\begin{proof}
Let 
$(i_{v_0})_*(\widetilde{\gamma})\in H_*(\M(\A'; \Gamma, G), \Z)$ 
be a torus cycle. 
If ${v_0}\notin H_{\alpha, V}$, 
then 
$(i_{{{v_0}}})_*(\widetilde{\gamma})$ is clearly contained in the image of 
the map $\iota$. If ${v_0}\in H_{\alpha, V}$, 
since $\M(\A\smallsetminus\A_{\tor}; \Gamma, V)$ is nonempty, 
there exists a small 
perturbation ${v_0'}$ of ${v_0}$ such that ${v_0'}
\in \M(\A\smallsetminus\A_{\tor}; \Gamma, V)$ 
(see Remark \ref{rem:dependence}). 
Then $H_*^{\torus}(\A(G))\ni 
(i_{{v_0'}})_*(\widetilde{\gamma})\longmapsto
(i_{{v_0}})_*(\widetilde{\gamma})\in
H_*^{\torus}(\A'(G))$. 
\end{proof}

\subsection{Meridian cycles}
\label{subsec:tmeridian}

The torus cycles introduced in the previous section are not enough to 
generate the homology group $H_*(\M(\A; \Gamma, G), \Z)$. We also need to 
consider meridians 
of $H_{\alpha, G}$ to generate $H_*(\M(\A; \Gamma, G), \Z)$. 

Let us first recall the notion of layers. 
A layer of $\A(G)$ is a connected component of a non-empty intersection 
of elements of $\A(G)$. Let $\scS\subset\A$. By 
Proposition \ref{prop:intersect}, every connected component 
of $H_{\scS, G}:=\bigcap_{\alpha\in\scS}H_{\alpha, G}$ is isomorphic to 
\begin{equation*}
\left((S^1)^p\times\R^q\right)^{r_{\Gamma}-r_{\scS}}. 
\end{equation*}
We sometimes call the number $r_{\scS}$ the rank of the layer. 
Since $H_{\emptyset, G}=\Hom(\Gamma, G)$, a connected component of 
$\Hom(\Gamma,  G)$ is a layer of rank $0$. Similarly, a connected component 
of $H_{\alpha, G}$ for $\alpha\in\A\smallsetminus\A_{\tor}$ is a layer 
of rank $1$. 

Let $L$ be a layer. 
Denote the set of $\alpha$ such that $H_{\alpha, G}$ contains $L$ by 
$\A_L:=\{\alpha\in\A\mid L\subset H_{\alpha, G}\}$, and the contraction by 
$\A^L:=\A/\A_L$. Note that $L$ can be considered to be a rank $0$ layer 
of $\A^L(G)$. Define 
\begin{equation*}
\begin{split}
\M^L(\A)
:=& 
L\smallsetminus\bigcup_{H_{\alpha, G}\not\supset L} H_{\alpha, G}\\
=& L\cap\M(\A^L; \Gamma/\langle\A_L\rangle, G). 
\end{split}
\end{equation*}
(We consider $\M(\A^L; \Gamma/\langle\A_L\rangle, G)$ as a subset 
of $\Hom(\Gamma, G)$ as in Proposition \ref{prop:disjoint}.) 

Let $L_1\subset\Hom(\Gamma, G)$ be a rank $1$ layer of $\A(G)$, and 
let $L_0$ be the rank $0$ layer that contains $L_1$. 
We wish to define the meridian homomorphism 
\begin{equation*}
\mu_{L_0/L_1}^{\varepsilon}: 
H_*(\M^{L_1}(\A), \Z)\longrightarrow
H_{*+\varepsilon\cdot(g-1)}(\M^{L_0}(\A), \Z), 
\end{equation*}
where $g=\dim G=p+q>0$ and $\varepsilon\in\{0, 1\}$. 

Since the normal bundle of $L_1$ in $L_0$ is trivial, 
there is a tubular neighborhood 
$U$ of $\M^{L_1}(\A)$ in $L_0$ such that 
$U\simeq \M^{L_1}(\A)\times D^g$ with the identification 
$\M^{L_1}(\A)=\M^{L_1}(\A)\times\{0\}$, 
where $D^g$ is the $g$-dimensional disk. 
Then 
$U\cap \M^{L_0}(\A)\simeq \M^{L_1}(\A)\times D^{g*}$, where 
$D^{g*}=D^g\smallsetminus\{0\}$. We denote the corresponding inclusion 
by $i:\M^{L_1}(\A)\times D^{g*}\hookrightarrow M^{L_0}(\A)$. 
For a given $\gamma\in H_*(\M^{L_1}(\A), \Z)$, define the element 
$\mu_{L_0/L_1}^{\varepsilon}(\gamma)\in H_{*+\varepsilon\cdot(g-1)}
(\M^{L_0}(\A), \Z)$ as follows. 
\begin{itemize}
\item[(0)] 
For $\varepsilon=0$, 
let $p_0\in D^{g*}$. Then 
$\gamma\times [p_0]\in H_*(\M^{L_1}(\A))\otimes H_0(D^{g*})\subset
H_{*}(\M^{L_1}(\A)\times D^{g*})$, and 
$\mu_{L_0/L_1}^{0}(\gamma):=i_*(\gamma\times [p_0])$. 
\item[(1)] 
For $\varepsilon=1$, 
let $S^{g-1}\subset D^{g*}$ be a sphere of small radius. Then 
$\gamma\times [S^{g-1}]\in H_*(\M^{L_1}(\A))\otimes H_{g-1}(D^{g*})\subset
H_{*+g-1}(\M^{L_1}(\A)\times D^{g*})$ (this part is essentially the 
Gysin homomorphism). 
Now define $\mu_{L_0/L_1}^{1}(\gamma):=i_{*}(\gamma\times [S^{g-1}])$. 
\end{itemize}
Similarly, we can define the meridian map 
\begin{equation*}
\mu_{L_j/L_{j+1}}^{\varepsilon}: 
H_*(\M^{L_{j+1}}(\A), \Z)\longrightarrow
H_{*+\varepsilon\cdot(g-1)}(\M^{L_j}(\A), \Z) 
\end{equation*}
between layers $L_j\supset L_{j+1}$ with consecutive ranks. 

\begin{definition}
\label{def:mercyc}
A cycle $\gamma\in H_d(\M(\A; \Gamma, G), \Z)$ is called a \emph{meridian cycle} 
if there exists some $k\geq 0$ and 
\begin{itemize}
\item[(a)] 
a flag $L_0\supset L_1\supset\cdots\supset L_k$ of layers with
$\rank L_j=j$, such that 
$L_0\cap\M(\A; \Gamma, G)\neq\emptyset$ (or equivalently, 
$L_0\subset\M(\A_{\tor}; \Gamma, G)$), 
\item[(b)] 
a sequence $\varepsilon_1, \dots, \varepsilon_k\in\{0, 1\}$, and 
\item[(c)] 
a torus cycle $\tau\in H_{d-(g-1)\sum_{i=1}^k\varepsilon_i}
(\M^{L_k}(\A), \Z)$,  
\end{itemize}
such that 
\begin{equation*}
\gamma=
\mu_{L_0/L_1}^{\varepsilon_1}\circ
\mu_{L_1/L_2}^{\varepsilon_2}\circ\cdots\circ
\mu_{L_{k-1}/L_k}^{\varepsilon_k}(\tau). 
\end{equation*}
We call the minimum such $k$ the depth of $\gamma$. 
\end{definition}

By definition, a meridian cycle of depth $0$ is a torus cycle 
of a connected component $L_0$. 
Furthermore, a cycle $\gamma\in H_*(\M(\A; \Gamma, G), \Z)$ is a 
meridian cycle of depth $k>0$ if and only if there exist layers 
$L_0\supset L_1$ of rank $0$ and $1$ respectively, with 
$L_0\cap\M(\A; \Gamma, G)\neq\emptyset$, 
$\varepsilon\in\{0, 1\}$ and a meridian cycle 
$\gamma'\in H_{*-\varepsilon\cdot(g-1)}(\M^{L_1}(\A), \Z)$ 
of depth $(k-1)$ such that 
$\gamma=\mu_{L_0/L_1}^{\varepsilon}(\gamma')$. 

Note that 
in Definition \ref{def:mercyc}, $\M^{L_0}(\A)$ is a non-empty open subset 
of $\M(\A; \Gamma, G)$. Hence we have the induced injection 
$H_*(\M^{L_0}(\A), \Z)\hookrightarrow H_*(\M(\A; \Gamma, G), \Z)$. 
We denote by $H_*^{\mer}(\A(G))$ the submodule of 
$H_*(\M(\A; \Gamma, G), \Z)$ generated by the images of meridian cycles. 
It is clear that 
\begin{equation*}
H_*^{\torus}(\A(G))\subset
H_*^{\mer}(\A(G))\subset
H_*(\M(\A; \Gamma, G), \Z). 
\end{equation*}

\begin{lemma}
\label{lem:surjmerid}
Assume that $q>0$. 
Let $\alpha\in\A\smallsetminus\A_{\tor}$, and let 
$\A':=\A\smallsetminus\{\alpha\}$. Then 
\begin{equation}
\label{eq:surjmerid}
H_*^{\mer}(\A(G))\longrightarrow H_*^{\mer}(\A'(G))
\end{equation}
is surjective. 
\end{lemma}

\begin{proof}
We prove this by induction on $\#(\A\smallsetminus\A_{\tor})$ 
and the depth $k$ of the meridian cycle $\gamma$. 
If $\#(\A\smallsetminus\A_{\tor})=1$, then $\A'=\A_{\tor}$. 
In this case, the meridian cycles of $\M(\A'; \Gamma, G)$ are torus cycles, and 
the result follows from Lemma \ref{lem:surjtorus}. 
Now assume that $\#(\A\smallsetminus\A_{\tor})>1$. 
Let $\gamma\in H_*(\M(\A'; \Gamma, G), \Z)$ be a 
meridian cycle of $\A'$. Suppose that $\gamma$ can be expressed as $\gamma=
\mu_{L_0/L_1}^{\varepsilon_1}\circ\cdots\circ
\mu_{L_{k-1}/L_k}^{\varepsilon_k}(\tau)$, 
as in Definition \ref{def:mercyc}. 
If $k=0$, then $\gamma=\tau$ is a torus cycle. Hence, again by 
Lemma \ref{lem:surjtorus}, $\gamma$ is contained in the image 
of the map \eqref{eq:surjmerid}. We may therefore assume that $k>0$. 
Let $\gamma'=\mu_{L_1/L_2}^{\varepsilon_2}\circ\cdots\circ
\mu_{L_{k-1}/L_k}^{\varepsilon_k}(\tau)$. 
Then $\gamma'\in H_{*-\varepsilon_1\cdot(g-1)}^{\mer}((\A')^{L_1}(G))$ 
is a meridian cycle of depth $(k-1)$. 

We separate the proof into two cases depending on whether 
$\overline{\alpha}$ is a loop in $\A^{L_1}$. 

Suppose that $\overline{\alpha}$ is a loop in $\A^{L_1}$. 
Then $H_{\alpha, G}$ 
either contains $L_1$ or does not intersect $L_1$. 
In either case, $\M^{L_1}(\A)=\M^{L_1}(\A')$. 
Hence the 
tubular neighborhood $U$ of $\M^{L_1}(\A')$ satisfies 
$U\cap \M^{L_0}(\A')=U\cap \M^{L_0}(\A)=U\smallsetminus L_1$, 
and the meridian cycle 
$\gamma=\mu_{L_0/L_1}^{\varepsilon}(\gamma')$ can be constructed in 
$\M^{L_0}(\A)\subset\M^{L_0}(\A')$. Consequently, $\gamma$ is contained in 
the image $H_*^{\mer}(\A(G))\longrightarrow H_*^{\mer}(\A'(G))$. 

Suppose that $\overline{\alpha}$ is not a loop in $\A^{L_1}$. 
Then, by the induction hypothesis, 
there exists a meridian cycle $\widetilde{\gamma}'\in H_{*-\varepsilon_1\cdot(g-1)}^{\mer}(\A^{L_1}(G))$ 
that is sent to $\gamma'$ by the induced map 
\begin{equation*}
H_{*-\varepsilon_1\cdot(g-1)}^{\mer}(\A^{L_1}(G)) \longrightarrow H_{*-\varepsilon_1\cdot(g-1)}^{\mer}((\A')^{L_1}(G)), 
\widetilde{\gamma}'\longmapsto\gamma'. 
\end{equation*}
Using the following commutative diagram, we can conclude that 
$\gamma$ is also contained in the image: 
\begin{equation*}
\begin{CD}
\widetilde{\gamma}'\in H_{*-\varepsilon_1\cdot(g-1)}^{\mer}(\A^{L_1}(G)) @>>> H_{*-\varepsilon_1\cdot(g-1)}^{\mer}((\A')^{L_1}(G))\ni\gamma'\\
@V\mu_{L_0/L_1}^{\varepsilon_1}VV @VV\mu_{L_0/L_1}^{\varepsilon_1}V\\
H_*^{\mer}(\A(G)) @>>> H_*^{\mer}(\A'(G))\ni\gamma. 
\end{CD}
\end{equation*}
\end{proof}

\subsection{Mayer-Vietoris sequences and Poincar\'e polynomials}
\label{subsec:mv}

For simplicity, in this section, we will set 
$\M(\A):=\M(\A; \Gamma, G)$, 
$\M(\A'):=\M(\A'; \Gamma, G)$, and 
$\M(\A''):=\M(\A''; \Gamma'', G)$. 

\begin{theorem}
\label{thm:inductive}
Let $\A$ be a finite list of elements in a 
finitely generated abelian group $\Gamma$, and let 
$G=(S^1)^p\times \R^q\times F$, where $F$ is a finite abelian group. 
Assume that $q>0$, and set $g=\dim G=p+q$. Then the following hold. 
\begin{itemize}
\item[(i)] 
$H_*(\M(\A), \Z)$ is generated by meridian cycles. That is 
$H_*(\M(\A), \Z)=H_*^{\mer}(\A(G))$, 
and furthermore it is torsion free. 
\item[(ii)] 
If $\alpha$ is not a loop, then 
$H_*(\M(\A), \Z)\longrightarrow H_*(\M(\A'), \Z)$ is surjective. 
\item[(iii)] 
Let $\alpha\in\A$. Then 
\begin{equation*}
P_{\M(\A)}(t)=
\left\{
\begin{array}{ll}
P_{\M(\A')}(t)-P_{\M(\A'')}(t), &\ \mbox{ if $\alpha$ is a loop, }\\
P_{\M(\A')}(t)+t^{g-1}\cdot P_{\M(\A'')}(t), &\ \mbox{ if $\alpha$ is not a loop.}\\
\end{array}
\right. 
\end{equation*}
\end{itemize}
\end{theorem}

\begin{proof}
We first note that when $\alpha$ is a loop, 
$\M(\A')=\M(\A)\sqcup \M(\A'')$ is a decomposition into 
disjoint open subsets. Thus 
(iii) is obvious when $\alpha$ is a loop. 

We prove the other results by induction on $\#(\A\smallsetminus\A_{\tor})$. 
If $\A=\A_{\tor}$, then (i) follows from 
\[
H_*(\M(\A), \Z)=H_*^{\mer}(\A(G))=H_*^{\torus}(\A(G))
\]
(see \eqref{eq:torconn} in \S \ref{subsec:tcycle}),  
and there is nothing to prove for (ii) and (iii). 

Assume that 
$\A\smallsetminus\A_{\tor}\neq\emptyset$, and 
suppose that $\alpha\in\A\smallsetminus\A_{\tor}$. 
Let $U$ be a tubular neighborhood of $\M(\A'')$ in $\M(\A')$, as 
in \S \ref{subsec:tmeridian}. 
Set $U^*:=U\cap\M(\A)\simeq \M(\A'')\times D^{g*}$. 
Consider the Mayer-Vietoris sequence 
associated with the covering 
$\M(\A')=U\cup\M(\A)$. We have the following diagram: 
\begin{equation*}
{\footnotesize 
\begin{CD}
@>>> H_k(U^*) @> f_k >> H_k(U)\oplus H_k(\M(\A)) @> g_k >> 
H_k(\M(\A'))@>>>\\
@. @AA h_1 A @AA h_2 A @AA h_3 A\\
@. H_k^{\mer}(U^*) @>f'_k>> H_k^{\mer}(U)\oplus H_k^{\mer}(\A(G))@> g'_k >> 
H_k^{\mer}(\A'(G)), 
\end{CD}
}
\end{equation*}
where $H_*^{\mer}(U^*)=H_*^{\mer}(\A''(G))\otimes H_*(D^{g*})$ and 
$H_k^{\mer}(U)\simeq H_k^{\mer}(\A''(G))$. 
The first line is a part of the Mayer-Vietoris long exact sequence. 
The vertical arrows $h_1, h_2$ and $h_3$ are the inclusion of the subgroup 
generated 
by meridian cycles. By the induction hypothesis, 
$h_1$ and $h_3$ are isomorphic. 
Lemma \ref{lem:surjmerid} implies that $g_k'$ is surjective. 
Hence, $g_k$ is also surjective. The surjectivity of $g_{k+1}$ 
implies that $f_k$ is injective. Therefore, the long exact sequence 
breaks into short exact sequences. 
The torsion freeness follows immediately. 
Thus 
\begin{equation*}
\rank H_k(U)+\rank H_k(\M(\A))=\rank H_k(U^*)+\rank H_k(\M(\A')), 
\end{equation*}
which implies the inductive formula (iii). 
A diagram chase shows that 
$h_2$ is also surjective. Hence 
$H_*(\M(\A), \Z)=H_*^{\mer}(\A(G))$. 
\end{proof}

If $G=(S^1)^p\times \R^q\times F$ as in the previous theorem, 
the Poincar\'e polynomial of $G$ is 
\begin{equation*}
P_G(t)=(1+t)^p\times\#F. 
\end{equation*}
We can compute the Poincar\'e polynomial of the complement $\M(\A)$ 
using $P_G(t)$ and the $G$-Tutte polynomial $T_{\A}^G(x, y)$. 

\begin{theorem}
\label{thm:poin}
Let $G$ be a non-compact abelian Lie group with finitely many 
connected components. Set $g=\dim G$. Then 
\begin{equation}
\label{eq:poin}
\begin{split}
P_{\M(\A)}(t)
&=
P_G(t)^{r_{\Gamma}-r_\A}\cdot t^{r_\A(g-1)}\cdot 
T_{\A}^{G}
\left(
\frac{P_G(t)}{t^{g-1}}+1, 0
\right)\\
&=
(-t^{g-1})^{r_{\Gamma}}\cdot
\chi_{\A}^G\left(
 -\frac{P_G(t)}{t^{g-1}}
\right). 
\end{split}
\end{equation}
\end{theorem}

\begin{proof}
We prove the result by induction on $\#\A$. Suppose that $\A=\emptyset$. Then 
$\M(\A)=\Hom(\Gamma, G)\simeq \Hom(\Gamma_{\tor}, G)\times G^{r_\Gamma}$, 
and $\chi_{\A}^G(t)=\#\Hom(\Gamma_{\tor}, G)\times t^{r_\Gamma}$. 
The formula \eqref{eq:poin} follows immediately. 

Suppose $\A\neq\emptyset$. Then, using Corollary \ref{cor:del-cont2} and 
Theorem \ref{thm:inductive} (iii), 
Formula \eqref{eq:poin} can be proved by induction. 
\end{proof}

\begin{remark}
\label{rem:recovering}
Theorem \ref{thm:poin} recovers the known formulas 
\eqref{eq:bjorner} and \eqref{eq:torpoin}. 
\end{remark}

\begin{remark}
\label{rem:noncpt}
If $G$ is a compact group, then formula \eqref{eq:poin} does not hold 
unless $\A=\emptyset$. There are several steps that fail 
for compact groups. For example the surjectivity of torus cycles 
(Lemma \ref{lem:surjtorus}) fails, so the proof of the surjectivity of 
meridian cycles (Lemma \ref{lem:surjmerid}) does not work. 
Furthermore, the existence of the fundamental class is an obstruction for 
breaking the Mayer-Vietoris sequence into short exact sequences. 
\end{remark}

\section{Relationship with arithmetic matroids}

\label{sec:matroid}

In this section, we discuss the relationship between $G$-multiplicities 
and arithmetic matroid structures.

\subsection{Properties of $G$-multiplicities}

\newcommand{\scC}{\mathcal{C}}
\newcommand{\scR}{\mathcal{R}}

We summarize the construction of the dual of a representable 
arithmetic matroid. Let $\A=\{\alpha_1, \dots, \alpha_n\}$ be a 
finite list of elements in a finitely generated abelian group $\Gamma$. 
In \cite{dad-mo}, 
D'Adderio and Moci constructed 
another finitely generated abelian group $\Gamma^\dagger$ and a 
list $\A^\dagger=\{\alpha_1^\dagger, \dots, \alpha_n^\dagger\}$ of elements 
in $\Gamma^\dagger$ labelled by the same index set $[n]=\{1, \dots, n\}$ 
(see \cite[\S 3.4]{dad-mo} for details). Let us recall the construction 
briefly. Assume that $\Gamma$ can be expressed as 
$\Gamma=\Z^m/\langle\bm{v}_1, \dots, \bm{v}_h\rangle$. Choose representatives 
$\widetilde{\alpha}_i\in\Z^m$ of $\alpha_i\in\Gamma$. 
Define 
\begin{equation*}
\Gamma^\dagger:=\Z^{n+h}/
\langle {}^t(\widetilde{\alpha}_1, \dots, \widetilde{\alpha}_n, \bm{v}_1, \dots, \bm{v}_h)\rangle, 
\end{equation*}
where the denominator is the subgroup generated by $m$ columns 
of the $(n+h)\times m$ matrix 
${}^t(\widetilde{\alpha}_1, \dots, \widetilde{\alpha}_n, \bm{v}_1, \dots, \bm{v}_h)$. 
Let $\bm{e}_i$ be the standard basis of $\Z^{n+h}$. Set 
$\alpha_i^\dagger:=\overline{\bm{e}_i}\in\Gamma^\dagger$ for $i=1, \dots, n$. 
Now we have the list 
$\A^\dagger=\{\alpha_1^\dagger, \dots, \alpha_n^\dagger\}$. 
For a subset $S\subset [n]$, we have (\cite[\S 3.4]{dad-mo}) 
\begin{equation}
\label{eq:dual}
\begin{split}
r_S^\dagger
&=\#S-r_{[n]}+r_{S^c},\\
(
\Gamma^\dagger/
\langle\alpha_i^\dagger\mid i\in S\rangle
)_{\tor}
&\simeq
(
\Gamma/
\langle\alpha_i\mid i\in S^c\rangle
)_{\tor}, 
\end{split}
\end{equation}
where $S^c=[n]\smallsetminus S$, 
$r_S=\rank\langle\alpha_i\mid i\in S\rangle$ and 
$r_S^\dagger=\rank\langle\alpha_i^\dagger\mid i\in S\rangle$ 
(the second relation in \eqref{eq:dual} is not a 
canonical isomorphism). 
Note that $\A^\dagger$ has rank $r_{\A^\dagger}=\#\A-r_{\A}$. 

Let $G$ be a torsion-wise finite abelian group. 
Recall from Definition \ref{def:multiplicit} that 
$m(\scS; G):=\#\Hom\left((\Gamma/\langle\scS\rangle)_{\tor}, G\right)$ for any 
$\scS \subset \A$. 

Denote the $G$-multiplicity of $(\Gamma^\dagger, \A^\dagger)$ by 
\begin{equation*}
m^\dagger(S; G):=\#\Hom
\left(
(
\Gamma^\dagger/
\langle\alpha_i^\dagger\mid i\in S\rangle
)_{\tor}, G
\right). 
\end{equation*}
The second relation in \eqref{eq:dual} implies that 
\begin{equation*}
m^\dagger(S; G)=m(S^c; G). 
\end{equation*}
The operation $(-)^\dagger$ is reflexive in the sense that 
\begin{equation*}
\begin{split}
r_S&=\#S-r_{[n]}^\dagger+r_{S^c}^\dagger, \\
m(S; G)&=m^\dagger(S^c; G), 
\end{split}
\end{equation*}
and $G$-Tutte polynomials satisfy 
\begin{equation}
\label{eq:dualGTutte}
T_{\A^\dagger}^G(x, y)=T_{\A}^G(y, x). 
\end{equation}

\begin{theorem}
\label{thm:G-matroid}
The $G$-multiplicities satisfy the following four properties (we borrow 
the numbering from \cite[\S 2.3]{dad-mo}). 
\begin{itemize}
\item[(1)]  If $\scS \subset \A$ and $\alpha \in \A$ satisfy $r_{\scS \cup \{\alpha\}} = r_{\scS}$, then $m(\scS \cup \{\alpha\} ;G)$ divides $m(\scS ;G)$.

\item[(2)]  If $\scS \subset \A$ and $\alpha \in \A$ satisfy $r_{\scS \cup \{\alpha\}} = r_{\scS}+1$, then $m(\scS ;G)$ divides $m(\scS \cup \{\alpha\} ;G)$.

\item[(4)] If $\scS \subset  \scT \subset \A$ and $r_{\scS} = r_{\scT}$, then
\begin{equation*}
\rho_{\scT}(\scS; G):= \sum_{\scS \subset  \underset{\bullet}{\scB} \subset \scT} (-1)^{\#\scB - \#\scS}m(\scB ;G) \ge 0.
\end{equation*}

\item[(5)] If $\scS \subset  \scT \subset \A$ and $r_{\scT} = r_{\scS}+\#(\scT\smallsetminus\scS)$, then
\begin{equation*}
\label{Axiom5}
\rho^*_{\scT}(\scS; G):= \sum_{\scS \subset  \underset{\bullet}{\scB} \subset \scT} (-1)^{\#\scT - \#\scB}m(\scB ;G) \ge 0.
\end{equation*}
\end{itemize}
Additionally, if $G$ is a (torsion-wise finite) divisible abelian group, 
that is, the multiplication-by-$k$ map $k:G\longrightarrow G$ is 
surjective for 
any positive integer $k$, then the $G$-multiplicities satisfy the following. 
\begin{itemize}
\item[(3)] 
If $\scS \subset  \scT \subset \A$ and $\scT$ is a disjoint union $\scT= \scS \sqcup \scB \sqcup \scC$ such that for all $\scS \subset \scR \subset  \scT$, we
have $r_{\scR}=r_{\scS} + \#(\scR \cap \scB)$, then
\begin{equation*}
m(\scS ;G) \cdot m(\scT ;G) = m(\scS \sqcup \scB ;G) \cdot 
m(\scS \sqcup \scC ;G).
\end{equation*}
\end{itemize}
\end{theorem}

\begin{proof}
Property (1) follows from the fact that there exists a group epimorphism $(\Gamma/\langle\scS\rangle)_{\tor} \longrightarrow (\Gamma/\langle\scS \cup\{\alpha\}\rangle)_{\tor}$ (\cite[Lemma 5.2]{br-mo}), and by applying the functor $\Hom(\textendash, G)$ to this epimorphism. 

By the above construction, $(r^\dagger, m^\dagger)$ satisfies (1), 
which is equivalent to property (2) for $(r, m)$. 

We prove (4) by showing that $\rho_{\scT}(\scS; G)$ is the cardinality of a 
certain finite set. 
Property (4) is clearly true if $\scS= \scT$, so assume that $\scS\subsetneq \scT$. Let us define $\Gamma'$ by 
\begin{equation*}
\Gamma':=\{g\in\Gamma\mid\exists n>0\mbox{ such that }n\cdot g\in\langle\scS\rangle\}. 
\end{equation*}
It is also characterized by 
$(\Gamma/\langle\scS\rangle)_{\tor}=\Gamma'/\langle\scS\rangle$. 
By the assumption $r_{\scS}=r_{\scT}$, we have 
$\scS\subset\scT\subset\Gamma'$. 
If $\scS\subset\scB\subset\scT$, we also have 
$(\Gamma/\langle\scB\rangle)_{\tor}=\Gamma'/\langle\scB\rangle$. 
Therefore, 
$\Hom((\Gamma/\langle\scB\rangle)_{\tor}, G)=
\Hom(\Gamma'/\langle\scB\rangle, G)$ 
can be considered as a subset of 
$\Hom((\Gamma/\langle\scS\rangle)_{\tor}, G)=
\Hom(\Gamma'/\langle\scS\rangle, G)$. 
By the principle of inclusion-exclusion and Proposition \ref{prop:intersect}, 
we have 
\begin{equation*}
\begin{split}
\rho_{\scT}(\scS; G)
&=
\sum_{\scS \subset  \underset{\bullet}{\scB} \subset \scT}
(-1)^{\#\scB-\#\scS}\cdot m(\scB; G)
\\
&=
\sum_{\scS \subset  \underset{\bullet}{\scB} \subset \scT}
(-1)^{\#\scB-\#\scS}\cdot\#\Hom(\Gamma'/\langle\scB\rangle, G)
\\
&=
\#\M(\scT/\scS; \Gamma'/\langle\scS\rangle, G), 
\end{split}
\end{equation*}
which is clearly non-negative.

We can prove (5) by an argument similar to that for (2) by using duality. 

Finally, to prove property (3) we generalize the argument used in \cite[Lemma 2.6]{dad-mo}. 
We consider the following diagram composing of two short exact sequences:
\begin{equation*}
{\footnotesize 
\begin{CD}
0 @>>> \left(\dfrac{\Gamma}{\langle \scS \sqcup \scC \rangle}\right)_{\tor} @>>> \left(\dfrac{\Gamma}{\langle \scT \rangle}\right)_{\tor}@>>> 
\ \left(\dfrac{\Gamma}{\langle \scT \rangle}\right)_{\tor}/\left(\dfrac{\Gamma}{\langle \scS \sqcup \scC \rangle}\right)_{\tor} 
@>>> 0\\
@. @. @. @AA \simeq A\\
0 @>>> \left(\dfrac{\Gamma}{\langle \scS \rangle}\right)_{\tor} @>>> \left(\dfrac{\Gamma}{\langle \scS \sqcup \scB \rangle}\right)_{\tor}@>>> \left(\dfrac{\Gamma}{\langle \scS \sqcup \scB \rangle}\right)_{\tor}/\left(\dfrac{\Gamma}{\langle \scS \rangle}\right)_{\tor} @>>> 0. 
\end{CD}
}
\end{equation*} 
(The isomorphism indicated by the vertical arrow is proved in 
\cite[Lemma 5.3]{br-mo}.) 
Since $G$ is divisible, $G$ is an injective $\Z$-module and the functor $\Hom(\textendash, G)$ is exact.
Applying the functor $\Hom(\textendash, G)$ to the diagram we obtain property (3).
\end{proof}

\begin{remark}
\label{rem:counterex}
When $G$ is a connected abelian Lie group, that is, $G=(S^1)^p\times\R^q$, $G$ is a torsion-wise finite and divisible group. 
It is easily seen that  property (3) fails in many cases. For example, let $\Gamma:=\Z^2, \scS:=\{(0,2)\}, \scB:=\{(2,1)\}, \scC:=\{(0,1)\}$ and $G:=\Z/2\Z$.
Then 
$
(\Gamma/\langle S\rangle)_{\tor}\simeq\Z/2\Z, 
(\Gamma/\langle S\cup B\rangle)_{\tor}\simeq\Z/4\Z, 
(\Gamma/\langle S\cup C\rangle)_{\tor}\simeq\{0\}, 
(\Gamma/\langle T\rangle)_{\tor}\simeq\Z/2\Z
$, 
and 
$m(\scS ;G) \cdot m(\scT ;G) = 4 \ne 2 = m(\scS \sqcup \scB ;G) \cdot m(\scS \sqcup \scC ;G).$
\end{remark}

\subsection{(Non-)positivity of coefficients} 
\label{subsec:nonpos}

As was proved in \cite[Theorem 3.5]{moci-tor}, the arithmetic Tutte polynomial 
$T_{\A}^{\arith}(x, y)$ is a polynomial with positive coefficients. 
In this section, we show that the $G$-Tutte polynomial has positive 
coefficients for some special cases. 
However, for a general group $G$, we show that the $G$-Tutte polynomial 
can have negative coefficients by exhibiting an explicit example. 

\begin{theorem}
Let $G$ be a torsion-wise finite divisible abelian group. Then the coefficients of the $G$-Tutte polynomial $T_{\A}^G(x, y)$ are positive integers.
\end{theorem}
\begin{proof}
When $G$ is a torsion-wise finite divisible group, the pair $(\Gamma, \A)$ together with the $G$-multiplicities form an arithmetic matroid. It is proved in \cite[Theorem 4.5]{br-mo} that the coefficients of the arithmetic Tutte polynomial of a pseudo-arithmetic matroid (and hence of an arithmetic matroid) are positive integers.
\end{proof}

\begin{proposition}
\label{prop:positivity}
Let $\Gamma$ be a finitely generated abelian group, and let $G$ be a 
torsion-wise finite group. 
\begin{itemize}
\item[(i)] 
If $\A\subset\Gamma$ consists of loops (i.e., $\A\subset
\Gamma_{\tor}$), then $T_{\A}^{G}(x, y)$ has positive coefficients. 
\item[(ii)] 
If $\A\subset\Gamma$ consists of coloops (i.e., $r_{\A}=\#\A$), 
then $T_{\A}^{G}(x, y)$ has positive coefficients. 
\end{itemize}
\end{proposition}
\begin{proof}
(i) follows immediately from Theorem \ref{thm:finpos}. 
(ii) follows immediately from (i) and 
\eqref{eq:dualGTutte} (note that if $\A$ consists of coloops, then 
$\A^\dagger$ consists of loops). 
\end{proof}
In general, the $G$-Tutte polynomial can have negative coefficients 
as in the next example. 
\begin{example}
\label{ex:negative}
Let $\Gamma=\Z\oplus\Z/4\Z$, let $\A=\{\alpha, \beta\}$ with 
$\alpha=(2, \overline{1})$ and $\beta=(0, \overline{2})\in\Gamma$, and 
let $G=\Z/4\Z$. Then by direct computation, we have 
\begin{equation*}
T_{\A}^G(x, y)=2xy+2x+2y-2. 
\end{equation*}
This also produces a counter-example to axiom (P) of 
Br\"and\'en-Moci \cite[\S 2]{br-mo}. With notation 
in \cite[\S 2]{br-mo}, $[\emptyset, \A]$ is a molecule, and 
\begin{equation*}
\rho(\emptyset, \A; G)=(-1)\cdot\sum_{\scB\subset\A}(-1)^{2-\#\scB}
m(\scB; G)=-2<0. 
\end{equation*}
Therefore, in this case, the multiplicity 
$m(\scB; G)$ ($\scB\subset\A$) does not form a pseudo-arithmetic matroid 
in the sense of Br\"and\'en-Moci \cite[\S 2]{br-mo}. 
\end{example}

\subsection{Convolution formula}

\label{subsec:convol}

The following result is the $G$-Tutte polynomial version of 
the so-called convolution formula \cite{eti-ver, krs}. 

\begin{theorem}
\label{thm:convol}
Let $\A\subset\Gamma$ be a list in a finitely generated group $\Gamma$, and 
let $G_1$ and $G_2$ be torsion-wise finite groups. Then 
\begin{equation}
\label{eq:convol}
T_{\A}^{G_1\times G_2}(x, y)=
\sum_{\underset{\bullet}{\scB}\subset\A}
T_{{\scB}}^{G_1}(0, y)\cdot
T_{\A/\scB}^{G_2}(x, 0). 
\end{equation}
\end{theorem}
\begin{proof}
The right-hand side of the formula is equal to 
\begin{equation*}
\begin{split}
&
\sum_{\underset{\bullet}{\scB}\subset\A}
\left\{
\sum_{\underset{\bullet}{\scS}\subset\scB}
m(\scS; G_1)(-1)^{r_{\scB}-r_{\scS}}(y-1)^{\#\scS-r_{\scS}}
\right\}\\
&
\times\left\{
\sum_{\scB\subset\underset{\bullet}{\scT}\subset\A}
m(\scT; G_2)(x-1)^{r_{\A}-r_{\scB}-(r_{\scT}-r_{\scB})}
(-1)^{\#\scT-\#\scB-(r_{\scT}-r_{\scB})}
\right\}\\
=&
\sum_{\underset{\bullet}{\scS}\subset
\underset{\bullet}{\scB}\subset
\underset{\bullet}{\scT}\subset\A}
m(\scS; G_1)m(\scT; G_2)
(x-1)^{r_{\A}-r_{\scT}}(y-1)^{\#\scS-r_{\scS}}
(-1)^{\#\scT-\#\scB-r_{\scT}-r_{\scS}}\\
=&
\sum_{\underset{\bullet}{\scS}=
\underset{\bullet}{\scB}=
\underset{\bullet}{\scT}\subset\A}
m(\scS; G_1)m(\scS; G_2)
(x-1)^{r_{\A}-r_{\scS}}(y-1)^{\#\scS-r_{\scS}}\\
&
+
\sum_{\underset{\bullet}{\scS}\subsetneq
\underset{\bullet}{\scT}\subset\A}
\left\{
m(\scS; G_1)m(\scT; G_2)
(x-1)^{r_{\A}-r_{\scT}}(y-1)^{\#\scS-r_{\scS}}
\sum_{{\scS}\subset
\underset{\bullet}{\scB}
\subset\scT}
(-1)^{\#\scT-\#\scB-r_{\scT}-r_{\scS}}
\right\}. 
\end{split}
\end{equation*}
The first term is 
equal to $T_{\A}^{G_1\times G_2}(x, y)$ from 
the multiplicativity $m(\scS; G_1\times G_2)=
m(\scS; G_1)m(\scS; G_2)$ (see Proposition \ref{prop:finitetf}). 
The second term vanishes because, when 
$\scS\subsetneq\scT$, we have 
$\sum_{{\scS}\subset
\underset{\bullet}{\scB}
\subset\scT}
(-1)^{\#\scB}=0$. 
\end{proof}

The classical convolution formula \cite{eti-ver, krs} for matroids 
representable over $\Q$ is obtained from Theorem \ref{thm:convol} by 
replacing $G_1$ and $G_2$ by $\{0\}$. 
Theorem \ref{thm:convol} can also be specialized to the 
Backman-Lenz \cite{ba-le} convolution formula 
when $G_1\times G_2=S^1 \times \{0\}$ or $\{0\}\times S^1$.

\medskip

\noindent
\textbf{Acknowledgements:} 
YL is partially supported by JSPS KAKENHI Grant Number 16J00125. 
TNT gratefully acknowledges the support of the scholarship program of 
the Japanese Ministry of Education, Culture, Sports, Science, and Technology 
(MEXT) under grant number 142506. 
MY is partially supported by 
JSPS KAKENHI Grant Number JP15KK0144, JP16K13741, JP18H01115 
and Humboldt Foundation. 
The authors thank Emanuele Delucchi for sharing 
observations on arithmetic Tutte and characteristic quasi-polynomials, 
Luca Moci for suggestions concerning the convolution formula, and 
Roberto Pagaria for pointed out some mistakes in the previous version. 
The authors also thank ALTA group in Bremen University for their hospitality 
where part of this work was done.


\begin{thebibliography}{999}

\bibitem{ar-ca-he}
F. Ardila, F. Castillo, M. Henley, 
The arithmetic Tutte polynomials of the classical root systems. 
\emph{Int. Math. Res. Not.} \textbf{IMRN 2015}, no. 12, 3830-3877. 

\bibitem{ath-adv}
C. A. Athanasiadis, 
Characteristic polynomials of subspace arrangements and finite fields. 
\emph{Adv. Math.} \textbf{122} (1996), no. 2, 193--233. 

\bibitem{ba-le}
S. Backman, M. Lenz, 
A convolution formula for Tutte polynomials of arithmetic matroids and 
other combinatorial structures. Preprint, 
arXiv:1602.02664

\bibitem{bbc}
C. Bajo, B. Burdick, S. Chmutov, 
On the Tutte-Krushkal-Renardy polynomial for cell complexes. 
\emph{J. Combin. Theory Ser. A} \textbf{123} (2014), 186-201.

\bibitem{bpr}
S. Basu, R. Pollack, M. -F. Roy, 
Algorithms in real algebraic geometry. Second edition. 
Algorithms and Computation in Mathematics, 10. 
\emph{Springer-Verlag, Berlin}, 2006. x+662 pp.

\bibitem{bergv}
O. Bergvall, 
Cohomology of Complements of Toric Arrangements Associated to Root Systems. 
Preprint, arXiv:1601.01857 

\bibitem{bibby}
C. Bibby, 
Cohomology of abelian arrangements. 
\emph{Proc. Amer. Math. Soc.} \textbf{144} (2016), no. 7, 3093-3104. 

\bibitem{bjo}
A. Bj\"orner, 
Subspace arrangements. 
\emph{First European Congress of Mathematics, Vol. I (Paris, 1992)}, 321-370, 
Progr. Math., 119, Birkh\"auser, Basel, 1994. 

\bibitem{bl-sa}
A. Blass, B. Sagan, 
Characteristic and Ehrhart polynomials. 
\emph{J. Algebraic Combin.} \textbf{7} (1998), no. 2, 115-126. 

\bibitem{br-mo}
P. Br\"and\'en, L. Moci, 
The multivariate arithmetic Tutte polynomial. 
\emph{Trans. Amer. Math. Soc.} \textbf{366} (2014), no. 10, 5523-5540. 

\bibitem{cddmp}
F. Callegaro, M. D'Adderio, E. Delucchi, L. Migliorini, R. Pagaria, 
Orlik-Solomon-type presentations for the cohomology algebra of toric arrangements. 
Preprint, arXiv:1806.02195v1


\bibitem{car-spo}
S. Caracciolo, A. Sportiello, 
General duality for abelian-group-valued statistical-mechanics models. 
\emph{J. Phys. A} \textbf{37} (2004), no. 30, 7407-7432. 

\bibitem{cos-ras}
M. Coste, Real Algebraic Sets. 
\emph{Arc spaces and additive invariants in real algebraic and analytic 
geometry}, 1-32, Panor. Synth\`eses, 24, Soc. Math. France, Paris, 2007.

\bibitem{dad-mo}
M. D'Adderio, L. Moci, 
Arithmetic matroids, the Tutte polynomial and toric arrangements. 
\emph{Adv. in Math.} \textbf{232} (2013) 335-367. 

\bibitem{dec-pro}
C. De Concini, C. Procesi, 
On the geometry of toric arrangements. 
\emph{Transform. Groups} \textbf{10} (2005), no. 3-4, 387-422. 

\bibitem{de-mo}
E. Delucchi, L. Moci, 
Colorings and flows on CW complexes, Tutte quasi-polynomials and 
arithmetic matroids. 
Preprint, arXiv:1602.04307

\bibitem{de-ri}
E. Delucchi, S. Riedel, 
Group actions on semimatroids. 
Preprint, arXiv:1507.06862

\bibitem{eti-ver}
G. Etienne, M. Las Vergnas, 
External and internal elements of a matroid basis. 
\emph{Discrete Math.} \textbf{179} (1998), no. 1-3, 111-119.

\bibitem{fink-moci}
A. Fink, L. Moci, 
Matroids over a ring. 
\emph{J. Eur. Math. Soc. } \textbf{18} (2016), no. 4, 681-731. 

\bibitem{gor-mac}
M. Goresky, R. MacPherson, 
Stratified Morse Theory, in: Ergeb. Math. Grenzgeb., Vol. 14, 
Springer-Verlag, Berlin, 1988.

\bibitem{hum}
J. E. Humphreys, 
Reflection groups and Coxeter groups. 
Cambridge Studies in Advanced Mathematics, 29. 
\emph{Cambridge University Press, Cambridge}, 1990. xii+204 pp.

\bibitem{ktt-cent}
H. Kamiya, A. Takemura, H. Terao, 
Periodicity of hyperplane arrangements with integral coefficients 
modulo positive integers. 
\emph{J. Algebraic Combin.} \textbf{27} (2008), no. 3, 317--330. 

\bibitem{ktt-quasi}
H. Kamiya, A. Takemura, H. Terao, 
The characteristic quasi-polynomials of the arrangements of 
root systems and mid-hyperplane arrangements. 
\emph{Arrangements, local systems and singularities}, 177--190, 
\textbf{Progr. Math., 283}, \emph{Birkh\"auser Verlag, Basel}, 2010. 

\bibitem{krs}
W. Kook, V. Reiner, D. Stanton, 
A convolution formula for the Tutte polynomial. 
\emph{J. Combin. Theory Ser. B} \textbf{76} (1999), no. 2, 297-300.

\bibitem{leh-tor}
G. I. Lehrer, 
A toral configuration space and regular semisimple conjugacy classes. 
\emph{Math. Proc. Cambridge Philos. Soc.} \textbf{118} (1995), no. 1, 105-113. 

\bibitem{loo-coh}
E. Looijenga, 
Cohomology of $\mathcal{M}_3$ and $\mathcal{M}_3^1$. 
\emph{Mapping class groups and moduli spaces of Riemann surfaces 
(G\"ottingen, 1991/Seattle, WA, 1991)}, 205-228, 
Contemp. Math., 150, Amer. Math. Soc., Providence, RI, 1993.

\bibitem{moci-lincei}
L. Moci, 
Combinatorics and topology of toric arrangements defined by root systems. 
\emph{Atti Accad. Naz. Lincei Rend. Lincei Mat. Appl.} \textbf{19} (2008), 
no. 4, 293-308. 

\bibitem{moci-tor}
L. Moci, 
A Tutte polynomial for toric arrangements. 
\emph{Trans. Amer. Math. Soc.} \textbf{364} (2012), no. 2, 1067-1088.

\bibitem{os}
P. Orlik, L. Solomon, 
Combinatorics and topology of complements of hyperplanes. 
\emph{Invent. Math.} {\bf 56} (1980), 167--189. 

\bibitem{ot} 
P. Orlik, H. Terao, Arrangements of hyperplanes. 
Grundlehren der Mathematischen Wissenschaften, 300. 
Springer-Verlag, Berlin, 1992. xviii+325 pp.

\bibitem{oxley}
J. Oxley, 
Matroid theory. Second edition. Oxford Graduate Texts in Mathematics, 21. 
Oxford University Press, Oxford, 2011. xiv+684 pp. 

\bibitem{sokal}
A. D. Sokal, 
The multivariate Tutte polynomial (alias Potts model) for graphs and matroids. 
\emph{Surveys in combinatorics 2005}, 173-226, 
London Math. Soc. Lecture Note Ser., 327, 
Cambridge Univ. Press, Cambridge, 2005. 

\bibitem{sut}
R. Suter, 
The number of lattice points in alcoves and the exponents of the 
finite Weyl groups. 
\emph{Math. Comp.} \textbf{67} (1998), no. 222, 751--758. 

\bibitem{thi}
M. B. Thistlethwaite, 
A spanning tree expansion of the Jones polynomial. 
\emph{Topology} \textbf{26} (1987), no. 3, 297-309. 

\bibitem{welsh}
D. J. A. Welsh, 
Complexity: knots, colourings and counting. 
London Mathematical Society Lecture Note Series, 186. 
Cambridge University Press, Cambridge, 1993. viii+163 pp.

\bibitem{yos-wor}
M. Yoshinaga, 
Worpitzky partitions for root systems and characteristic quasi-polynomials. 
\emph{Tohoku Mathematical Journal} \textbf{70} (2018) 39-63. 

\bibitem{yos-linial}
M. Yoshinaga, 
Characteristic polynomials of Linial arrangements for exceptional 
root systems. 
\emph{Journal of Combinatorial Theory, Series A}. 
\textbf{157} (2018) 267-286. 

\bibitem{zas-face}
T. Zaslavsky, 
Facing up to arrangements: Face-count formulas for 
partitions of space by hyperplanes. 
\emph{Memoirs Amer. Math. Soc.} 1 (1975), no. 154, vii+102 pp. 

\end{thebibliography}
\end{document}